\documentclass[11pt]{article}
\usepackage[utf8]{inputenc}
\usepackage{soul,enumitem,comment,subcaption}
\usepackage{pgfplots,tikz}
\usepackage{kbordermatrix}
\renewcommand{\kbldelim}{(}
\renewcommand{\kbrdelim}{)}
\usepackage[normalem]{ulem}

\newcommand{\required}[1]{\section*{\hfil \sharsokol2013advanced\hfil}}

\newcommand{\supp}{\operatorname{supp}}

\newcommand{\wf}{\operatorname{WF}}
\usepackage{verbatim,url}
\usepackage{graphicx, epsfig}
\graphicspath{{helical_figs/}}
\usepackage{bbm}
\usepackage{amssymb,amsmath,amsfonts,amsthm}
\usepackage{upref}
\newcommand{\be}{\begin{equation}}
\newcommand{\ee}{\end{equation}}
\newcommand{\bs}{\begin{split}}
\newcommand{\beq}{\begin{equation*}}
\newcommand{\eeq}{\end{equation*}}
\newcommand{\bal}{\begin{align}}
\newcommand{\eal}{\end{align}}

\newcommand{\s}{\mathcal S}
\newcommand{\q}{\nu}
\newcommand{\ed}{\tau}

\newcommand{\cC}{\mathcal C}
\newcommand{\WF}{\mathrm{WF}}  

\newcommand{\CL}{\mathcal L}
\newcommand{\CP}{\mathcal P}
\newcommand{\CR}{\mathcal R}

\newcommand{\br}{\mathbb{R}}

\newcommand{\pa}{\partial}
\newcommand{\bt}{\beta}
\newcommand{\al}{\alpha}
\newcommand{\la}{\lambda}

\newcommand{\ioi}{\int_0^{\infty}}

\newcommand{\dd}{\mathrm{d}}

\newcommand{\ga}{\gamma}  
\newcommand{\ts}{\tilde s}

\def\red#1{\textcolor[rgb]{1, 0, 0}{#1} }

\newcommand{\e}{\epsilon}
\newcommand{\de}{\delta}

\numberwithin{equation}{section}

\newcommand{\bel}[1]{\be\label{#1}}

\newcommand{\smo}{\setminus \mathbf{0}}

\newcommand{\norm}[1]{\left\lVert#1\right\rVert}      
\newcommand{\abs}[1]{\left|#1\right|}                 
\newcommand{\paren}[1]{\left(#1\right)}               
\newcommand{\sparen}[1]{\left\{#1\right\}}      

\newcommand{\vy}{\mathbf{y}}
\newcommand{\vx}{\mathbf{x}}
\newcommand{\vv}{\mathbf{v}}
\newcommand{\vz}{\mathbf{z}}
\newcommand{\vu}{\mathbf{u}}

\newcommand{\Cc}{\mathcal{C}}

\newcommand{\Dc}{\mathcal{D}}
\newcommand{\Ec}{\mathcal{E}}
\newcommand{\Fc}{\mathcal{F}}
\newcommand{\Sc}{\mathcal{S}}
\newcommand{\vsig}{{\boldsymbol{\sigma}}}

\newcommand{\rr}{{{\mathbb R}}}

\newcommand{\rn}{{{\mathbb R}^n}}

\newtheorem{theorem}{Theorem}[section]
\newtheorem{lemma}[theorem]{Lemma}
\newtheorem{proposition}[theorem]{Proposition}
\newtheorem{definition}[theorem]{Definition}
\newtheorem{corollary}[theorem]{Corollary}
\newtheorem{assumption}[theorem]{Assumption}

\newcommand{\partyf}[2]{\frac{\partial #2}{\partial y_{#1}}}

\theoremstyle{definition}

\newtheorem{remark}{Remark}

\newcommand{\tblu}{\textcolor{blue}}

\usepackage{hyperref}
\hypersetup{
    colorlinks=true,
    linkcolor=blue,
    filecolor=magenta,      
    urlcolor=cyan,
}

\usepackage{xcolor}

\usepackage{authblk} 

\title{Microlocal analysis of non-linear artifacts in cone beam CT}
\author[$\dagger$]{James W. Webber} 
\author[$*$]{Alexander Katsevich}

\affil[$\dagger$]{\small (Corresponding author) Cleveland Clinic Lerner College of Medicine of Case Western Reserve University School of Medicine, USA\par webberj5@ccf.org}

\affil[$*$]{\small School of Data, Mathematical, and Statistical Sciences, University of Central Florida, USA\par alexander.katsevich@ucf.edu}
\date{}
\begin{document}
\maketitle

\begin{abstract} 
We present a novel microlocal analysis of beam hardening artifacts arising in cone-beam X-ray CT, where the set of X-ray sources is restricted to a 1D smooth curve $ \gamma\subset \mathbb{R}^3$. 
We show that, when the CT data is modeled in the standard way using the Beer Lambert law, the exponential term in the model creates singularities in the data that are not present in the linear X-ray transform.
We assume that the attenuation coefficient $\mu$ (the reconstruction target) has a jump discontinuity across a surface $\s\subset\br^3$, and is smooth otherwise. We prove that these additional singularities in the data occur when the X-ray beam is tangent to $\s$ at two points simultaneously. To investigate how the singularities in the data propagate to the reconstruction space, we apply Filtered Back Projection (FBP) type reconstruction. We prove that the artifacts due to beam hardening are locally of conormal type, and lie on a 2-D surface which is the union of all double tangent rays which intersect $\gamma$. The artifacts are notably weaker than the reconstructed jumps of $\mu$ (i.e., the desired singularities), and we quantify this using the order of the corresponding conormal distributions. While our primary theory applies to the regions of $\s$ that are smooth, we also extend our theory to non-smooth $\s$ with ``ridges." These arise where $\s$ is locally the intersection of two smooth surface patches meeting transversely along a curve (e.g., the edge of a cuboid). In addition, we present simulated reconstructions of metal objects in circular cone-beam CT to validate our theory.
\end{abstract}

\section{Introduction}
In this paper, we present a novel microlocal analysis of beam hardening artifacts in non-linear cone-beam CT. Using the theory of conormal distributions, we describe precisely the type of artifacts that occur in the reconstruction due to non-linear effects, and we quantify their strength using the order of the distribution. In contrast to much of the previous literature, we do not need to limit the source spectrum and photon intensity to perform this analysis. We explain this novelty in more detail in our literature review below.
Moreover, our analysis considers the practically relevant case when the X-ray beams are constrained to pass through a 1-D smooth curve, $\gamma \subset \mathbb{R}^3$. This is the appropriate model for cone beam CT.  On the other hand, to the best of our knowledge, the relevant literature up to this point exclusively considers overdetermined ray transforms, e.g., when one has access to all ray integrals in $\br^3$. This has important implications on the microlocal analysis, as the restricted ray transform does not satisfy the Bolker condition globally. Thus, in order to describe the artifacts in an elegant way, and quantify their strength, we discover that we must impose certain non-restrictive geometric conditions, which we introduce and investigate here.

Typically, in X-ray CT applications, one would bombard a metal surface (e.g., a Tungsten target) with high energy electrons to generate a spectrum of photons, say in some energy range $[0,E_{\text{mx}}]$, where $E_{\text{mx}}$ is the maximum energy of the source. These are the foundations of polychromatic X-ray source tubes. Let $P(E)$ denote the intensity of the source beam at energy $E$ multiplied by the efficiency of the X-ray detector. Then, by the Beer-Lambert law, we can model the CT data corresponding to a given ray $L$ as follows,
\be\label{mod_intro}
g(L) =-\log\left[ \int_0^{E_{\mathrm{mx}}}P(E)\exp\paren{ -\int_L \mu(\vx,E) \mathrm{d}\vx } \mathrm{d}E\right].
\ee
where $\mu = \mu(\vx,E)$, $\vx \in \mathbb{R}^3$ is the spatial variable, and the photon energy is $E > 0$. Throughout this paper, we assume that $\mu$ has a jump discontinuity across a surface $\s\subset\br^3$, and is smooth otherwise. This means that $g$ also has singularities when $L$ is tangent to $\s$. If the dependence of $\mu$ on $E$ can be neglected (e.g., in the case of soft tissue within certain energy intervals), then the logarithm and exponential functions effectively cancel out, yielding a linear X-ray transform \cite{katsevich2026}. In the vast majority of CT scans, however, this is not the case and the forward model is intrinsically non-linear. We consider the non-linear case in this paper, specifically the artifacts which arise due to non-linear effects.

The literature considers microlocal analysis of artifacts due to non-linearities in the data \cite{palacios2018quantitative, chihara2024geodesic, katsevich2026,palamodov1986nonlinear,palamodov1990some,park2017characterization,chihara2022microlocal, wang2024identification, wang2021streak}. 
{The earliest papers are \cite{palamodov1986nonlinear,palamodov1990some}, where the analysis is mostly heuristic.} In \cite{palacios2018quantitative}, the authors consider beam hardening artifacts in classical 2-D parallel beam CT. In local neighborhoods of the sinogram which correspond to double tangent rays to $\s$, the authors show that, due to non-linear effects, the data is the product of two conormal distributions $u_1, u_2$ of the same order each with singularities on smooth transversally intersecting manifolds. In such cases, as has been established since Hormander \cite[Theorem 8.2.10]{hor}, $\WF(u_1 u_2)$ can contain elements that are not in $\WF(u_1)$ or $\WF(u_2)$. The singularities of the linear analog are described simply by $u_1 + u_2$, which cannot create new singularities in this way. The authors in \cite{palacios2018quantitative} describe such phenomena, and prove that these additional singularities created in the data induce streaking artifacts in the reconstruction.  They compute the order of the singularities, and show, e.g., that the streak artifacts are one degree smoother on Sobolev scale (i.e., one derivative smoother) when compared to the reconstructed edges of the ground truth. They also discuss methods to enhance the reconstruction to suppress the artifacts due to beam hardening. In cases when $\mu$ is separable (i.e., $\mu(\vx,E) = \mu_1(\vx)\mu_2(E)$), other methods exist in the literature to remove the artifacts which do not use microlocal theory \cite{joseph1978method}. This method is called ``soft tissue correction" or ``water pre-correction." 

In \cite{chihara2024geodesic}, the authors consider streaking artifacts in $n$-dimensions, which occur along families of double-tangent geodesics. When $n=2,3$ and the geodesics are straight lines, their model can be applied to X-ray CT. The linear counterpart of this model, the so-called ``$d$-plane" transform, is introduced and studied by the same authors in \cite{chihara2022microlocal}. In \cite{chihara2024geodesic} metal objects with smooth strictly convex boundary are considered. Similarly to \cite{palacios2018quantitative}, they show that the non-linearities in the data cause new singularities to occur, and they describe precisely how these propagate to the reconstruction using the theory of conormal distributions. As the authors have access to all geodesics through a given space, the data is overdetermined when $n \geq 3$, which helps formalize the analysis as the linear operators in this setting are typically well-behaved (e.g., the X-ray transform in $\mathbb{R}^3$).

{In {many of the works} cited above {(e.g., see \cite{park2017characterization, wang2024identification, wang2021streak}),} the authors pursue similar microlocal analysis ideas {to \cite{palacios2018quantitative,chihara2024geodesic}.} A recurring theme is to characterize the additional singularities in the data induced by nonlinearities and to explain how these give rise to artifacts along double-tangent rays in the reconstruction.}

{Another common feature of {the papers from the literature (see \cite{park2017characterization, palacios2018quantitative, wang2021streak, wang2024identification, chihara2022microlocal, chihara2024geodesic} for specific examples)} is the physical model.} {Namely,} they assume that the source spectrum, $P(E)$, is supported on a narrow interval $[E_0 - \e, E_0 +\e]$, where $E_0 \in [0,E_{\text{mx}}]$ and $\e > 0$ is some small bin width (e.g., $\e = 1$keV).  While this model allows for many interesting analyses, restricting the energy spectrum in this way has its drawbacks, and is not commonly done {in CT applications.} For example, to implement this model in practice, one could apply filters to the spectrum to suppress photon energies away from $E_0$. However, this greatly reduces photon counts and increases noise. 

The analysis of nonlinear artifacts using a more accurate physical model was first carried out in \cite{katsevich2026}. However, in this paper only 2-D parallel beam CT is addressed. We expand this theory to 3-D, and do not require such restrictions on the spectrum used in earlier literature. We require only that $P(E)$ be an integrable function on $[0,E_{\text{mx}}]$ which satisfies $\int_{[0,E_{\text{mx}}]} P(E) \mathrm{d}E > 0$, which is far less restrictive. 

We begin our analysis by considering the physical model \eqref{mod_intro}. Similar to \cite{palacios2018quantitative, chihara2024geodesic, katsevich2026,palamodov1986nonlinear,palamodov1990some,park2017characterization,chihara2022microlocal, wang2024identification, wang2021streak}, we show that when we work locally near points in the sinogram which correspond to double tangent rays, the data can be described as the product of two conormal distributions whose singularities lie on transversally intersecting smooth manifolds. The product generates singularities, and we show how these propagate to artifacts in FBP reconstruction. These are essentially caused by the exponential term in \eqref{mod_intro}, which, when the dependence of $\mu$ on $E$ cannot be neglected, is not addressed by the negative logarithm. We show specifically that the artifacts in the reconstruction are (locally) conormal distributions with singularities on a smooth 2-D surface, $\mathcal{L}_a$, which is the union of all double tangent rays which pass through $\gamma$. We quantify the strength of the artifacts using the order of the conormal distribution. 

As we are working with restricted ray transforms, we discover that certain geometric conditions are required in order to perform this analysis. These are discussed later in detail in Assumption \ref{ass:sngl}, one of which revolves around the Bolker condition. For example, in order for our theorems to hold, we essentially require that Bolker holds within the local neighborhoods we are considering, which is not always guaranteed as such restricted ray transforms do not satisfy Bolker globally. We explain here how such conditions are not restrictive for generic densities. That is, we can construct specific special cases where our assumptions do not hold, but for the purpose of practical application, our theory is generally valid. This is an important distinction between our work and, e.g., \cite{chihara2024geodesic}, which does not address restricted ray transforms in this way. 

While our central theory applies to local neighborhoods where $\s$ is smooth, we also generalize our theorems to cases where $\s$ has ``ridges." These arise when $\s$ is locally the intersection of two smooth surface patches meeting transversely along a curve, e.g., the edge (not including the vertices) of a cuboid. Similar analysis is also conducted in \cite{katsevich2026} but in 2-D. We generalize these results to cone-beam CT.

The remainder of this paper is organized as follows. In section \ref{sect:defns}, we state definitions and theorems from microlocal analysis and the literature on conormal distributions that will be needed to prove our theorems. In section \ref{phys_model}, we introduce our physical model for $\mu$ and the CT data, and we also revise known theory on the linear restricted ray transform which will be needed for our non-linear analysis. In section \ref{sec:sing tang}, we address the singularities in the data which correspond to single tangent rays. In this case, the data behaves locally as a conormal distribution analogous to the linear transform, and we calculate the order of the distribution. In section \ref{sec:dbl tang}, we consider double tangent rays, and show how the non-linearities create singularities in the data that are not present in the linear transform. In sections \ref{sec:props comp} and \ref{sec: sing recon}, we apply FBP and describe how the singularities in the data propagate to the image domain. In section \ref{sec:folds}, we generalize this theory to $\mu$ with non-smooth \tblu{$\s$}. To finish, we present simulated experiments to validate our theory in section \ref{sec:exp}.

\section{Definitions}\label{sect:defns} 

In this section, we review some theory from microlocal analysis which will be used in our theorems. We first provide some
notation and definitions.  Let $X$ and $Y$ be open subsets of
{$\br^{n_X}$ and $\br^{n_Y}$, respectively.}  Let $\Dc(X)$ be the space of smooth functions compactly
supported on $X$ with the standard topology and let $\mathcal{D}'(X)$
denote its dual space, the vector space of distributions on $X$.  Let
$\Ec(X)$ be the space of all smooth functions on $X$ with the standard
topology and let $\mathcal{E}'(X)$ denote its dual space, the vector
space of distributions with compact support contained in $X$. Finally,
let $\Sc(\rn)$ be the space of Schwartz functions, that are rapidly
decreasing at $\infty$ along with all derivatives. See \cite{Rudin:FA}
for more information. 


We now list some notation conventions that will be used throughout this paper:
\begin{enumerate}
\item For a function $f$ in the Schwartz space $\Sc(\br^{n_X})$, we {write
\[
\mathcal{F}f(\xi) = \int_{\br^{n_X}} e^{i x\cdot \xi}f(x)\ \dd x,\quad 
\mathcal{F}^{-1}\hat f(x) = \frac{1}{(2 \pi)^{n_x}}\int_{\br^{n_X}} e^{-i x\cdot \xi}f(\xi)\ \dd \xi
\]
for} the Fourier transform and inverse Fourier transform of $f$,
respectively, {and extend these in the usual way to tempered distributions $\Sc'(\br^{n_X})$} (see \cite[Section 7.1]{hor}). We have used the notation $\hat{f} = \mathcal{F}f$. 

\item We use the standard multi-index notation: if
$\al=(\al_1,\al_2,\dots,\al_n)\in \sparen{0,1,2,\dots}^{n_X}$
is a multi-index and $f$ is a function on $\br^{n_X}$, then
\[\partial^\al f=\paren{\frac{\partial}{\partial
x_1}}^{\al_1}\paren{\frac{\partial}{\partial
x_2}}^{\al_2}\cdots\paren{\frac{\partial}{\partial x_{n_X}}}^{\al_{n_X}}
f.\] If $f$ is a function of $(\vy,\vx,\vsig)$ then $\partial^\al_\vy f$ and $\partial^\al_\vsig f$ are defined similarly.

\item \label{item:T*Xident} We identify the cotangent
spaces of Euclidean spaces with the underlying Euclidean spaces. For example, the cotangent space, 
$T^*(X)$, of $X$ is identified with $X\times \br^{n_X}$. If $\Phi$ is a function of $(\vy,\vx,\vsig)\in Y\times X\times \rr^N$,
then we define $\dd_{\vy} \Phi = \paren{\partyf{1}{\Phi},
\partyf{2}{\Phi}, \cdots, \partyf{{n_X}}{\Phi} }$, and $\dd_\vx\Phi$ and $
\dd_{\vsig} \Phi $ are defined similarly. Identifying the cotangent space with the Euclidean space as mentioned above, we let $\dd\Phi =
\paren{\dd_{\vy} \Phi, \dd_{\vx} \Phi,\dd_{\vsig} \Phi}$.


\end{enumerate}

\noindent The singularities of a function and the directions in which they occur
are described by the wavefront set \cite[Definition 8.1.2]{hor}, which we now define.
\begin{definition}
\label{WF} Let $X$ be an open subset of $\br^{n_X}$ and let $f$ be a distribution in $\mathcal{D}'(X)$.  Let $(\vx_0,\xi_0)\in X\times (\br^{n_X}\setminus\{0\})$.  Then $f$ is \emph{smooth at $\vx_0$ in direction $\xi_0$} if there exist a neighborhood $U$ of $\vx_0$ and $V$ of $\xi_0$ such that for every $\Phi\in \Dc(U)$ and $N\in\br$ there exists a constant $C_N$ such that for all $\xi\in V$ {and $\lambda >1$},
\be
\left|\Fc(\Phi f)(\lambda\xi)\right|\leq C_N(1+\abs{\lambda})^{-N}.
\ee
The pair $(\vx_0,\xi_0)$ is in the \emph{wavefront set,} $\wf(f)$, if
$f$ is not smooth at $\vx_0$ in direction $\xi_0$. We define for fixed $\vx_0$, $\WF_{\vx_0}(f) = \{\xi\in \br^{n_X}\setminus\{0\}: (\vx_0,\xi) \in \WF(f)\}$, and we define the singular support of $f$, $\text{ssupp}(f)$, as the natural projection of $\WF(f)$ onto $X$.
\end{definition}

The wavefront set is an important consideration in imaging since elements of the wavefront set $\WF(f)$ will correspond to sharp features of an image $f$.

The wavefront set of a distribution on $X$ is normally defined as a
subset the cotangent bundle $T^*(X)$ so it is invariant under
diffeomorphisms, but we do not need this invariance, so we will
continue to identify $T^*(X) = X \times \rr^{n_X}$ and consider $\WF(f)$ as
a subset of $X\times (\rr^{n_X}\setminus\{0\})$.


 \begin{definition}[{\cite[Definition 7.8.1]{hor}}] \label{ellip}We define
 $S^m(Y \times X, \br^N)$ to be the
set of $a\in \Ec(Y\times X\times \br^N)$ such that for every
compact set $K\subset Y\times X$ and all multi--indices $\al,
\bt, \gamma$ the bound
\[
\left|\partial^{\gamma}_{\vy}\partial^{\bt}_{\vx}\partial^{\al}_{\vsig}a(\vy,\vx,\vsig)\right|\leq
C_{K,\al,\bt,\gamma}(1+\norm{\vsig})^{m-|\al|},\ \ \ (\vy,\vx)\in K,\
\vsig\in\br^N,
\]
holds for some constant $C_{K,\al,\bt,\gamma}>0$. 

 The elements of $S^m$ are called \emph{symbols} of order $m$.  Note
that this symbol class is  sometimes denoted $S^m_{1,0}$.  The symbol
$a\in S^m(Y \times X,\rr^N)$ is \emph{elliptic} if for each compact set
$K\subset Y\times X$, there is a $C_K>0$ and $M>0$ such that
\bel{def:elliptic} \abs{a(\vy,\vx,\vsig)}\geq C_K(1+\norm{\vsig})^m,\
\ \ (\vy,\vx)\in K,\ \norm{\vsig}\geq M.
\ee 
\end{definition}

\begin{definition}[{\cite[Section 7.8]{hor}, \cite[Definition 21.2.15]{hor3}}] \label{phasedef}
A real-valued function $\Phi=\Phi(\vy,\vx,\vsig)\in \Ec(Y\times X\times(\br^N\smo))$ is a \emph{phase
function} if $\Phi(\vy,\vx,\lambda\vsig)=\lambda\Phi(\vy,\vx,\vsig)$, $\forall \la>0$, and $\mathrm{d}\Phi$ is nowhere zero. The \emph{critical set of $\Phi$} is
\be
\Sigma_\Phi=\{(\vy,\vx,\vsig)\in Y\times X\times(\br^N\setminus\{0\}): \dd_{\vsig}\Phi=0\}.
\ee
 A phase function is
\emph{clean} if the critical set $\Sigma_\Phi$ is a smooth manifold {with tangent space defined {by} the kernel of $\mathrm{d}\,(\mathrm{d}_\sigma\Phi)$ on $\Sigma_\Phi$. Here, the derivative $\mathrm{d}$ is applied component-wise to the vector-valued function $\mathrm{d}_\sigma\Phi$. So, $\mathrm{d}\,(\mathrm{d}_\sigma\Phi)$ is treated as a Jacobian matrix of dimensions $N\times ({n_Y + n_X}+N)$.}
\end{definition}

\begin{definition}[{\cite[Definition 21.2.15]{hor3} and \cite[section 25.2]{hor4}}]\label{def:canon} 
Let $X$ and $Y$ be open subsets of $\rn$. Let $\Phi\in \Ec\paren{Y \times X \times{\rr}^N}$ be a clean phase function.  In addition, we say that $\Phi$ is nondegenerate if the differentials 
\be
\dd (\pa\Phi/\pa \sigma_j),\ j=1,2,\dots,N,
\ee
are linearly independent.
The canonical relation parameterized by $\Phi$ is defined as
\be\label{def:Cgenl} \begin{aligned} \Cc=&\sparen{
\paren{\paren{\vy,\dd_{\vy}\Phi(\vy,\vx,\vsig)};\paren{\vx,-\dd_{\vx}\Phi(\vy,\vx,\vsig)}}:(\vy,\vx,\vsig)\in
\Sigma_{\Phi}}{.}
\end{aligned}
\ee
\end{definition}

\begin{definition}\label{FIOdef}
Let $X$ and $Y$ be open subsets of $\br^{n_X}$ and $\br^{n_Y}$, respectively. Let an operator $A :
\Dc(X)\to \mathcal{D}'(Y)$ be defined by the distribution kernel $K_A\in \mathcal{D}'(Y\times X)$, in the sense that $Af(\vy)=\int_{X}K_A(\vy,\vx)f(\vx)\mathrm{d}\vx$. Then we call $K_A$ the \emph{Schwartz kernel} of $A$. A Fourier integral operator (FIO) of order $\mu = m + N/2 - (n_X+n_Y)/4$ is an operator $A:\Dc(X)\to \mathcal{D}'(Y)$ with Schwartz kernel given by an
oscillatory integral of the form
\be \label{oscint}
K_A(\vy,\vx)=\int_{\br^N}
e^{i\Phi(\vy,\vx,\vsig)}a(\vy,\vx,\vsig) \mathrm{d}\vsig,
\ee
where $\Phi$ is a clean nondegenerate phase function and $a$ is a symbol in $S^m(Y \times X , \br^N)$. The canonical relation of $A$ is the canonical relation $\mathcal{C}$ of $\Phi$ defined in
\eqref{def:Cgenl}. An FIO is called a pseudodifferential operator ($\Psi$DO) if $X = Y$ and its canonical relation $\Cc$ is contained in the diagonal, i.e.,
$\Cc \subset \Delta := \{ (\vx,\xi;\vx,\xi):\,(\vx,\xi)\in X\times(\br^N\setminus\{0\})\}$.
\end{definition}

FIOs are defined in \cite{hor4} more generally as operators with Schwartz kernel locally given by expressions of the form \eqref{oscint} where the local phase functions give rise to pieces of the same global cannonical relation. However, the local expression \eqref{oscint} is sufficient for our purposes.

We now state the H\"ormander-Sato Lemma \cite[Theorem 8.2.13]{hor} and \cite[Lemma 25.1.2]{hor4}, which provides  the relationship between the wavefront set of distributions and their images under FIO{s}.

\begin{theorem}[H\"ormander-Sato Lemma]\label{thm:HS} Let $f\in \Ec'(X)$ and
let ${A}:\Ec'(X)\to \Dc'(Y)$ be an FIO with canonical relation $\Cc$.
Then, $\wf({A}f)\subset \Cc(\wf(f))$.
\end{theorem}

Let $A$ be an FIO, then its formal adjoint $A^*$ is also an FIO, and if $\Cc$ is the canonical relation of $A$, then the canonical relation of $A^*$ is $\Cc^t$ \cite[Theorem 25.2.2]{hor4}. Many imaging techniques are based
on application of the adjoint operator $A^*$ and so to understand artifacts we consider $A^* A$ (or, if $A$ does not map to $\Ec'(Y)$, then $A^* \psi A$ for an appropriate smooth cutoff $\psi$). Because of Theorem \ref{thm:HS},
\begin{equation}
\label{compo}
\wf(A^* \psi A f) \subset (\Cc^t \circ \Cc)(\wf(f)).
\end{equation}
The next two definitions provide tools to analyze the composition in equation \eqref{compo}.

\begin{definition}
\label{defproj} Let $\Cc\subset T^*(Y\times X)$ be the canonical relation associated to the FIO ${A}:\mathcal{E}'(X)\to \mathcal{D}'(Y)$. We let $\Pi_L$ and $\Pi_R$ denote the natural left- and right-projections of $\Cc$, projecting onto the appropriate coordinates: $\Pi_L:\Cc\to T^*(Y)$ and $\Pi_R : \Cc\to T^*(X)$.
\end{definition}

The convention to call these projections left and right comes from the double fibration approach to the  Radon transforms, which are examples of FIOs \cite[Section 4]{gruhl1}.

If $A$ satisfies our next definition, then, for any properly supported $\Psi$DO $\mathcal{P}$, $A^* A$ (or $A^* \mathcal{P} A$) is a $\Psi$DO \cite{GS1977, quinto}.

\begin{definition}[Bolker condition]\label{def:bolker} 
Let ${A}:\Ec'(X)\to \Dc'(Y)$ be a FIO with canonical relation $\Cc$ then {$A$} (or $\Cc$) satisfies the Bolker condition if the natural projection $\Pi_L:\Cc\to T^*(Y)$ is an embedding (injective immersion).
\end{definition}

Using \eqref{compo}, we see that under the Bolker condition the wavefront set of $A^* \mathcal{P} Af$ will be contained in the wavefront set of $f$. Intuitively, the reconstructed image ($A^* \mathcal{P} Af$) will only include singularities at the same positions and in the same directions as the original image ($f$). 

We now define clean intersection \cite[Definition C.3.2]{hor3} and clean composition \cite[Definition 5.2, Chapter VIII]{trev2}.

\begin{definition}\label{def:clean int}
We say that two manifolds $L$ and $M$ intersect cleanly if $L\cap M$ is a manifold and at every point of $L\cap M$, $T(L\cap M) = (TL)\cap(TM)$.
\end{definition}

\begin{definition}\label{def:clean comp}
Let $\Cc \subset T^*X\backslash \{0\} \times T^*Y\backslash \{0\}$ and $\Lambda \subset T^*Y\backslash \{0\}$ be Lagrangian submanifolds. Then, we say the composition $\Cc \circ \Lambda$ is clean if $L = \Cc \times \Lambda$ and $M = (T^*X\backslash \{0\}) \times \Delta_Y$ intersect cleanly and the natural projection $L\cap M \to \Cc \circ \Lambda$ is proper. Here, $\Delta_Y$ denotes the diagonal in $T^*Y\backslash \{0\} \times T^*Y\backslash \{0\}$.
\end{definition}

We have the definition of Sobolev spaces from \cite[page 200]{natterer}.

\begin{definition}\label{sobo_spaces}
The Sobolev space of order $\al$ is the vector space of tempered distributions satisfying
\be
H^\al(\br^n) = \{f\in \mathcal{S}'(\br^n) : (1 + |\xi|^2)^{\al/2} \hat{f} \in L^2(\br^n)\}
\ee
with norm 
\be\label{HRnnorm}
\|f\|_{H^{\al}(\br^n)} = \|(1 + |\xi|^2)^{\al/2} \hat{f}\|_{L^2(\br^n)}.
\ee
\end{definition}

\noindent We now define local and microlocal Sobolev regularity and wavefront sets \cite{Q1993sing}.

\begin{definition} \label{def:SobolevWF}
A distribution $g$ is in $H^{\al}$ locally near a point $\vx_0$ if and only if there exists a cut-off function $\varphi \in C_c^{\infty}(\br^n)$ with $\varphi(\vx_0) \neq 0$ such that $\varphi g \in H^{\al}(\br^n)$. The distribution $g$ is in $H^{\al}$ microlocally near $(\vx_0,\xi_0)$ if and only if there is a cut-off function $\varphi \in C^{\infty}_c(\br^n)$ with $\varphi(\vx_0) \neq 0$ and function $u(\xi)$ homogeneous of degree zero and smooth on $\br^n\backslash \{0\}$ with $u(\xi_0)\neq 0$ such that $(1+|\xi|^2)^{\al/2}u(\xi)\mathcal{F}(\varphi g)(\xi) \in L^2(\br^n)$. We define the Sobolev wavefront set
\be
\WF^{\al}(g) = \left\{ (\vx,\xi) \in \WF(g) : g\ \text{is not in}\ H^{\al}\ \text{microlocally near}\ (\vx, \xi) \right\},
\ee
and $\text{ssupp}^{\al}(g)$ as the natural projection of $\WF^{\al}(g)$ onto $X$.
\end{definition}

We define the distribution $x_+ = 0 $ for $x < 0$ and $x_+ = x$ when $x\geq 0$. Similarly we define $x_- = (-x)_+$. More general distributions $x_\pm^r$, $r\not=-1,-2,\dots$, can be defined as well \cite[Section 3.2, Chapter I]{gs}. Related distributions are $(\lambda \pm i0)^r$. They can be expressed as \cite[eqs. (3), (4), Section 3.6, Chapter I]{gs}:
$$
(\lambda \pm i0)^r = \lambda_+^r + e^{\pm i\pi r } \lambda_-^r,\ r\not=-1,-2,\dots.
$$
From \cite[Entries 21 and 24, Table of Fourier Transforms]{gs}, we obtain the Fourier transform of the distributions $x_\pm^r$:
\be
\label{equ_dist}
\mathcal{F} (x_\pm^{-(r+1)}) = c(r)(\lambda \pm i0)^r,\ r\in\br,r\not=0,1,2,\dots,
\ee
where $\lambda$ is dual to $x$, $c(r) = \Gamma(-r)\exp(\mp i(\pi/2)r)$, and $\Gamma$ is the Gamma function. Thus, the Fourier transform of $x_\pm^r$ decays at a rate $r+1$ as $\lambda \to \infty$.

\noindent We now review some preliminaries on conormal distributions.
\begin{definition}[\cite{palacios2018quantitative,greenleaf1993recovering}]
\label{def_conormal}
Let $\Phi(\vx,\vsig)$, $\vsig\in\br^N$, be a phase function, let $a(\vx,\vsig)$ be a symbol, and let 
\be
N^*\Gamma = \{(\vx, \mathrm{d}_\vx\Phi(\vx,\vsig)) \in T^* X : \vx \in X, \mathrm{d}_\vsig\Phi(\vx,\vsig) = 0\},
\ee
be the conormal bundle of $\Gamma=\{\vx\in X:\dd \Phi(\vx,\vsig)=0\}$, a smooth manifold of codimension $N$ in $X$. Then, we denote $I^\beta(N^*\Gamma) $ to be the space of Lagrangian distributions of order $\beta$ associated with $N^*\Gamma$. The elements $u \in I^\beta(N^*\Gamma)$  take the form
\be
u = \int_{\br^N}a(\vx,\vsig) e^{i \Phi(\vx,\vsig)} \mathrm{d}\vsig,
\ee
where the order of $a$ is $\q = \beta + n_X/4 -N/2$. 

We will also sometimes equivalently write $u\in  \dot{I}^{\q}(\Gamma)$, where $\dot{I}^{\q}$ is the space of conormal distributions. Note that the order of $u$ is computed differently, depending on whether we consider the inclusion $u \in I^\beta(N^*\Gamma)$ or $u \in \dot{I}^\q(\Gamma)$, even though $I^\beta(N^*\Gamma)=\dot{I}^\q(\Gamma)$. We use the dot notation to distinguish the two spaces.
\end{definition}

We also have the following embedding theorem \cite[Theorem 5.4.1]{dusthorII_72}:
\begin{theorem}\label{thm:embedding}
$I^\beta(N^*\Gamma)\subset H^{\alpha}(\br^{n_X})$ if and only if $\bt+n_X/4+\al < 0$.
\end{theorem}

The below proposition will be used extensively throughout the paper in our analysis of the singularities of the non-linear transform.
\begin{proposition}
\label{prop_sqrt}
Let $s - \phi(\vv) = 0$ be a local defining equation of a smooth two-dimensional manifold $\Gamma\subset\br^3$, where $\phi$ is smooth, $s\in \br$ and $\vv \in \br^2$. For any $r >0$ and any smooth function $u$, one has
\be
\label{dist_order_0}
u(s,\vv)(s - \phi(\vv))^r_\pm \in \dot{I}^{-(r+1)}(\Gamma)=I^{-r-5/4}(N^*\Gamma).
\ee
\end{proposition}
\begin{proof}
By \eqref{equ_dist}, we have
\be
u(s,\vv)(s - \phi(\vv))^r_\pm = \tilde{u}(s,\vv)\int_{\br}(\lambda \pm i0)^{-(r+1)} e^{-i \lambda (s- \phi(\vv))} \dd\lambda,
\ee
where $\tilde{u}$ is smooth, and the coefficient $(2\pi)^{-1}$ is absorbed by $\tilde u$. Let $\varphi(\lambda)$ be a smooth cutoff which is equal to 1 in a small neighborhood of zero. Dropping the arguments of $u$ and $\tilde u$, we obtain
\be
\begin{split}
&u(s - \phi(\vv))^r_\pm\\ 
&= \tilde{u}\left[ \int_{\br}\varphi(\lambda)(\lambda \pm i0)^{-(r+1)} e^{-i \lambda (s- \phi(\vv))} \dd\lambda + \int_{\br}(1- \varphi(\lambda)) (\lambda \pm i0)^{-(r+1)} e^{-i \lambda (s- \phi(\vv))} \dd\lambda \right]\\
&= \tilde{u}\Big[ \psi_{\pm}(s- \phi(\vv)) + \int_{\br}a(\lambda) e^{i \lambda (s- \phi(\vv))} \dd\lambda\Big].
\end{split}
\ee
Here, $\psi_{\pm}$ is smooth by the Paley-Weiner theorem \cite{paley1934fourier}, and $a$ is a symbol order $\q = -(r+1)$, noting that we have localized away from the singularity in $(\lambda \pm i0)^{\q}$ at zero using the smooth cutoff. Hence, $u\in\dot{I}^{-(r+1)}(\Gamma)$. The order of the distribution ($\beta$) is calculated using $\q = \beta + 3/4 - N/2$, where $N = 1$ is the dimension of the variable $\lambda$. This yields $\beta = -r-1/2-3/4=-r-5/4$.
\end{proof}

We now define paired Lagrangian distributions.

\begin{definition}[\cite{greenleaf1993recovering}]
\label{pair_lag}
Let $\Lambda_1,\Lambda_2 \subset T^*X\backslash 0$ be a pair of cleanly intersecting Lagrangians. Then, associated to the pair $(\Lambda_1, \Lambda_2)$ is a class of Lagrangian distributions $u\in I^{p,l}(\Lambda_1, \Lambda_2)$, indexed by $p,l\in \br$, which satisfy $\WF(u) \subset \Lambda_1 \cup \Lambda_2$ and 
\be
I^{p,l}(\Lambda_1,\Lambda_2) \subset I^{p+l}(\Lambda_1 \backslash \Lambda_2),\quad   I^{p,l}(\Lambda_1,\Lambda_2) \subset I^p(\Lambda_2\setminus\Lambda_1).
\ee
\end{definition}

An important particular case arises when $Y_2 \subset Y_1 \subset X$ are smooth manifolds with $\text{codim}_X(Y_1) = d_1$ and $\text{codim}_X(Y_2) = d_1 + d_2$, for positive integers $d_1, d_2$. Then $N^*Y_1$ and $N^*Y_2$ intersect cleanly in codimension $d_2$. The space of distributions on $X$ conormal to the pair $(Y_1,Y_2)$ of orders $\q_1, \q_2$ is
\be\label{two types}
\begin{split}
\dot{I}^{\q_1,\q_2}(Y_1,Y_2)=I^{p,l}(N^*Y_2, N^* Y_1),\ p=\q_1 + \tfrac{d_1}{2}-\tfrac{n}{4},\ l=\q_2+\tfrac{d_2}{2}.
\end{split}
\ee

\noindent We now have the lemma from \cite[Lemma 1.1]{greenleaf1993recovering}.
\begin{lemma}
\label{lem_green}
Let $Y, Z \subset X$ be submanifolds, where $Y$ and $Z$ intersect transversally. Then, 
\be\label{prod distr}
\dot{I}^{\q_1}(Y)\cdot \dot{I}^{\q_2}(Z) \subset \dot{I}^{\q_1,\q_2}(Y,Y \cap Z)  + \dot{I}^{\nu_2,\q_1}(Z,Y \cap Z).
\ee
\end{lemma}
\noindent Later, when we present our main results, we will discover that the non-linearities in the data create distributional products, which cause additional singularities in the data that are not present in the linear analog. We will use Lemma \ref{lem_green} to quantify the strength of these added singularities in the data.

\section{Physical model}\label{phys_model}
In this section, we introduce our physical model and review some key microlocal properties of the (linear) restricted X-ray transform.

Let $\mu $ denote the attenuation coefficient of an imaging target. Generalizing somewhat the model used in \cite{katsevich2026}, we model $\mu$ as
\be
\label{mu_model}
\mu_E(\vx) = \mu(\vx,E) = \sum_{i = 1}^{m} \Phi_i(\vx,E) \chi_{\Omega_i}(\vx)\in L^2(X \times [0,E_{\mathrm{mx}}]),
\ee
where the $\Phi_i > 0$ are integrable functions in $E$ which are smooth in $\vx$.
Here, $E \in [0,E_{\mathrm{mx}}]$ represents the photon energy, and $E_{\mathrm{mx}} > 0$ is the maximum energy of the X-ray beam. We assume that the set $X$ is bounded. The $\Omega_i \subset X$ are simply connected compact domains with piecewise smooth boundaries $\partial \Omega_i$. The theory we develop later applies locally in neighborhoods where the boundary is smooth (and, in certain directions, convex). Once this core theory is established, we generalize to non-smooth boundaries with ``ridges" (e.g., the edge of a cuboid). We let $f = \sum_{i = 1}^{m} \chi_{\Omega_i}(\vx)$ be a function which describes the singularities of $\mu$ in $X$, and we let $\s = \text{ssupp}(f) = \cup_{i=1}^m \partial \Omega_i$.


The $\partial \Omega_i$ are assumed to be pairwise disjoint. Here, the characteristic functions account for the singularities of $\mu$, which are independent of $E$, and the $\Phi_i$ model smooth changes of the attenuation coefficient with respect to $\vx$ and $E$. Since $\Phi_i > 0$ and $\Phi_i$ is smooth in $\vx$, multiplying $\chi_{\Omega_i}$ by $\Phi_i$ in \eqref{mu_model} does not affect the singularities in $\vx$. 

Let $I \subset \br$ be a compact interval, and let $y(s) : I\to \br^3$ parameterize a 1-D curve, $\gamma$, of X-ray sources in $\br^3$ which is disjoint from $\text{supp}(f)$ and with $|y'(s)| \neq 0$. For example, when $y(s) = (R\cos(s),R\sin(s), hs/(2\pi))$ the source path is a helix with radius $R$ and helical pitch $h$ and $|y'(s)| = [R^2+(h/(2\pi))^2]^{1/2}$. We model the CT data using the non-linear X-ray transform
\be\label{model_data}
\mathcal{R}\mu(s,\Theta) =-\log\Big[ \int_0^{E_{\mathrm{mx}}}P(E)e^{-R\mu_E(s,\Theta)} \mathrm{d}E\Big],
\ee
where $P$ is a nonnegative integrable function on $[0,E_{\mathrm{mx}}]$ satisfying $\int_0^{E_{\mathrm{mx}}}P(E)\dd E > 0$.  The function $P$ represents the source spectrum multiplied by the detector efficiency, and
\be
Rf(s,\Theta) = \int_{\br} f(y(s) + t\Theta) \dd t
\ee
denotes the linear ray transform, where $\Theta \in S^2$ is the direction of the X-ray emitted from the source at $y(s)$. 

If the support of $f$ is on one side of $\ga$, i.e. no line $y(s)+t\Theta$, $t\in\br$, intersects $\mathrm{supp}(f)$ both when $t>0$ and $t<0$, then $Rf$ is equivalent to the cone beam transform $Rf(s,\Theta) = \ioi f(y(s) + t\Theta) \dd t$, which is the most common X-ray tomographic modality. Since we work locally, our results apply equivalently to the cone-beam and X-ray transform.

\subsection{Analysis of $R$}
In this section, we recall well-known microlocal analysis results on the linear restricted ray transform, $R$ (see, e.g., \cite{gruhl1}), that will be needed to analyze the singularities of the non-linear data, $\mathcal{R}\mu$.

The adjoint of $R$ is defined as follows
\be
R^*g(\vx) = \int_{\gamma} g\paren{s, \frac{\vx - y(s)}{|\vx - y(s)|}} \frac{1}{|\vx - y(s)|^2}\mathrm{d} s.
\ee
Let $\vy = (s,\Theta) \in Y = I \times S^2$ and let $L_\vy = \{y(s) + t\Theta : t\in \br\}$. In what follows, we parameterize $\Theta = \Theta(\al,\bt)$ so that the (co-)vectors 
$\Theta_\al= {\mathrm{d}_\al} \Theta$ and $\Theta_\bt= {\mathrm{d}_\bt} \Theta$ are linearly independent. 

Throughout this paper, we adopt the convention that covectors are row vectors, and (tangent) vectors are column vectors. For example, when taking the dot product of a covector $\vu$ and tangent vector $\vv$, we write this as $\vu\vv$, i.e., standard matrix multiplication. The dot product of two covectors (or tangent vectors) is written $\vu \cdot \vu$ using the dot notation.

We now have the result which is well-established in the microlocal literature. We also provide an expression for the canonical relation of $R$ in our chosen coordinate system.

\begin{theorem}[\cite{gruhl1}]\label{fio_thm}
$R$ is an FIO order $-1/2$ with canonical relation 
\be\label{CR}
\begin{split}
\mathcal{C} &= \Big(\big((s,\al,\bt),([\sigma_\al  \Theta_\al + \sigma_\bt\Theta_\bt] y'(s), -(\sigma_\al \Theta_{\al\al}+\sigma_\bt \Theta_{\bt\al},\sigma_\al \Theta_{\al\bt}+\sigma_\bt \Theta_{\bt\bt})(\vx - y(s)))\big);\\
& \big(\vx, \sigma_\al \Theta_\al + \sigma_\bt\Theta_\bt\big): s\in I,\vx\in X,(\sigma_\al ,\sigma_\bt)\in \br^2\backslash \{0\}, \Theta_\al (\vx - y(s)) = 0,   \Theta_\bt (\vx - y(s))= 0\Big).
\end{split}
\ee
\end{theorem}


\begin{proof}
We have
\be
\label{canon}
\begin{split}
Rf(s,\Theta) &= \int_{X} \delta(\Theta_\al(\vx - y(s)))\delta(\Theta_\bt(\vx - y(s))) f(\vx)\mathrm{d}\vx\\
&= \frac1{(2\pi)^2}\int_{\br^2}\int_{X} e^{-i \Phi(s,\Theta,\vx,\sigma_\al ,\sigma_\bt)} f(\vx)\mathrm{d}\vx\mathrm{d}\sigma_\al \mathrm{d}\sigma_\bt,
\end{split}
\ee
where $\supp(f)\subset X$, and $\Phi(s,\Theta,\vx,\sigma_\al ,\sigma_\bt) = [\sigma_\al \Theta_\al + \sigma_\bt \Theta_\bt](\vx - y(s))$ is the nondegenerate phase function (recall that $\Theta_\al$ and $\Theta_\bt$ are linearly independent). The expression for the canonical relation in these coordinates follows from Definition \ref{def:canon}. The order of the FIO is computed using Definition~\ref{FIOdef}: $\mathcal{O}(R) = 0+(2/2)-(6/4)=-1/2$.
\end{proof}

\noindent We introduce the notation
\be\label{useful matr}
\hat\Theta^\prime= \begin{pmatrix}\Theta_\al \\  \Theta_\bt\end{pmatrix},\
\hat\Theta^{\prime\prime}
=-\begin{pmatrix} \Theta\cdot\Theta_{\al\al} & \Theta\cdot\Theta_{\al\bt}\\
\Theta\cdot\Theta_{\al\bt} & \Theta\cdot\Theta_{\bt\bt}
\end{pmatrix},\ \vsig=(\sigma_\al,\sigma_\bt).
\ee 
Then we can rewrite \eqref{CR} in a simpler form
\be\label{CR alt}
\begin{split}
\cC
 = \Big(&((s,\Theta), ( \vsig\hat\Theta^\prime y'(s), t\vsig\hat \Theta^{\prime\prime}));(y(s)+t\Theta, \vsig\hat\Theta^\prime ):\\
&  s\in I,\Theta\in S^2,y(s)+t\Theta\in X,\vsig\in \br^2\backslash \{0\}\Big).
\end{split}
\ee

The following lemma illustrates how and where $\Cc$ fails to be a graph. In these cases, $R$ does not satisfy the Bolker condition, see Definition~\ref{def:bolker} \cite[Section 8]{flu}.

\begin{lemma}\label{lem:bolk fail} Suppose there exist two different points $(\vx_k,\xi_k)\in T^*\br^3\backslash \{0\}$, $k=1,2$, such that $(\vy,\eta):=\Cc(\vx_1,\xi_1)=\Cc(\vx_2,\xi_2)$. Then $\xi_1$ and $\xi_2$ are parallel, the three points $\vx_1$, $\vx_2$, and $y(s)$ lie on the same line, $\eta_s=0$, and $\xi_1\perp y^\prime(s)$.
\end{lemma}

\begin{proof}
We have to find two different $(\vx,\xi)\in T^*\br^3\backslash \{0\}$ that are mapped into the same $(\vy,\eta)\in T^*Y\backslash \{0\}$. From \eqref{CR alt}, $\vx=y(s)+t\Theta$ and $\xi=\sigma_\al \Theta_\al + \sigma_\bt\Theta_\bt$. Therefore we solve
\be\label{eta-xi}
\eta_s=\xi y'(s),\ \eta_\al=-[\sigma_\al  \Theta_{\al\al} + \sigma_\bt\Theta_{\bt\al}]\cdot t\Theta,\
\eta_\bt=-[\sigma_\al  \Theta_{\al\bt} + \sigma_\bt\Theta_{\bt\bt}]\cdot t\Theta.
\ee
for a triple $(t,\sigma_\al, \sigma_\bt)$, and the solution must be nonunique.

The last two equations in \eqref{eta-xi} represent a $2\times2$ system of linear equations in terms of  $t\sigma_\al$, $t\sigma_\bt$. Since $|\Theta(\al,\bt)|\equiv1$, we get that 
\be
\hat\Theta^{\prime\prime}=\begin{pmatrix} \Theta_\al\cdot\Theta_\al & \Theta_\al\cdot\Theta_\bt\\
\Theta_\al\cdot\Theta_\bt & \Theta_\bt\cdot\Theta_\bt
\end{pmatrix}.
\ee
By assumption, $\Theta_\al$ and $\Theta_\bt$ are linearly independent, therefore $\det\hat\Theta^{\prime\prime}\neq0$ and $t\sigma_\al$ and $t\sigma_\bt$ are computed uniquely.

It remains to determine the value of $t\not=0$ from the first equation in \eqref{eta-xi}.
The condition $t\neq0$ follows from the assumption $X\cap\ga=\varnothing$. We see that $t$ is not unique if and only if $\eta_s=\xi y'(s)=0$. All the assertions of the lemma now  follow. 
\end{proof}

The above proof implies that for any line $L_{\vy}$ and any vector $\xi$ orthogonal to both $\Theta$ and $y'(s)$, all points $\cC(\vx,\la\xi)$, $\vx\in L_{\vy}$, coincide, provided the scalar $\la=\la(\vx)$ is chosen appropriately.

Next we state the theorem which describes the singularities, including reconstruction artifacts, in FBP type reconstruction, see \cite{gruhl1, kat7, flu}.

\begin{theorem}
\label{lin_art}
Let $f \in \mathcal{E}'(X)$ and let $\Pi(\vx,\xi) = \{\vz \in \br^3 : \xi(\vz - \vx) = 0\}$. Then, $R^*Rf \in \mathcal{D}'(\br^3\backslash \gamma)$ and
\be
\WF(R^* \mathcal{P}Rf) \subset (\WF(f) \cap \mathcal{V}) \cup \Lambda,
\ee
where $\mathcal{P}$ is a properly supported $\Psi$DO, 
\be
\mathcal{V} = \{ (\vx,\xi) \in T^*X: \vx \in L_{\vy}, \vy \in Y, \xi \Theta = 0,\xi y^\prime(s)\not=0\}
\ee
is the set of visible singularities and
\be
\label{lambda}
\Lambda = \{ (\vx,\xi) \in T^*X : \xi y'(s)=0, \xi \Theta=0, \vx,\vz \in L_{\vy}, (\vz, \xi) \in \WF(f),\vy\in Y\}
\ee
is the set of artifacts that arise because of the violation of the Bolker condition.
\end{theorem}

The result follows essentially using Lemma \ref{lem:bolk fail} in combination with the Hormander-Sato Lemma, noting that the canonical relation of $R^*$ is $\Cc^t$, i.e., the transpose relation of $R$, and also noting that the canonical relation of $\mathcal{P}$ is a subset of the diagonal. Here, $\mathcal{P}$ represents a filter. For example, one can apply a suitable combination of derivatives in $\al$ and $\bt$ to enhance singularities, e.g., $\mathcal{P} = \pa_\al^2 +\pa_\bt^2$ or other similar operators \cite{luma, kat7, kat06a}. Note that not all reconstructions are of the type considered in the theorem, see e.g. \cite{Katsevich2002, Katsevich2004a}.

\section{Analysis of singularties due to non-linearity - single tangent lines}\label{sec:sing tang}
Now that we have established our model and {reviewed known} results on the linear transform, in this section, we present the first of our main results which address how the non-linearities in the CT data affect singularities. 

Pick some $\vx_0\in\s$ where $\s$ is smooth, and let $F(\vx)=0$ be a local equation of $\s$ near $\vx_0$. We assume that $\dd F(\vx)\neq0$ in a neighborhood of $\vx_0$. Let $\vy_0 = (s_0,\Theta_0) \in Y$ be such that $L_{\vy_0}$ is tangent to $\s$ at $\vx_0$. Recall that the normal curvature of $\s$ along $L_{\vy_0}$ at $\vx_0$ is not zero if $\Theta_0^{\top}F^{\prime\prime}(\vx_0)\Theta_0\neq 0$. Let $B_{\vy_0}(\delta) = \{\vy\in Y:|\vy - \vy_0| \leq \delta\}$ be a small neighborhood of $\vy_0$, where $\de>0$ is sufficiently small. Denote 
\be\label{Ga def}
\Gamma=\{\vy\in B_{\vy_0}(\delta):\ \text{the line $L_{\vy}$ is tangent to $\s$}\}.
\ee
Clearly, we locally have $\big(\cup_i \text{ssupp}(R\chi_i)\big)\cap B_{\vy_0}(\delta)\subset \Gamma$.

\begin{assumption}\label{ass:sngl}$ $
Suppose the line $L_{\vy_0}$ is tangent to $\s$ at $\vx_0$ where $\s$ is smooth.
\begin{enumerate}[label=(\arabic*)]
\item\label{sng curv} The normal curvature of $\s$ along $L_{\vy_0}$ is not zero. 
\item\label{sng transv} The plane $\Pi(\vx_0,\xi)$ is not tangent to $\ga$ at $y(s_0)$, where $\xi\in N_{\vx_0}^*\s$.
\end{enumerate}
\end{assumption}


\begin{lemma}\label{lem:Gamma_sm}
Suppose Assumption~\ref{ass:sngl} holds and $\vx_0$ is the only point where $L_{\vy_0}$ is tangent to $\s$. Then $\Gamma$ is a smooth 2-D manifold in a neighborhood of $\vy_0$ and its local equation can be written in the form $s=\phi(\Theta)$. Moreover, $\mathcal{C}:N^*\s\to N^*\Gamma$ is a local diffeomorphism. 
\end{lemma}

\begin{proof}
The desired equation for $\Gamma$ is obtained by solving 
\be\label{tang cond}
F(y(s)+t\Theta)=0,\ \dd F(y(s)+t\Theta)\Theta=0
\ee
for $s$ and $t$ in terms of $\Theta$. Here, $\mathrm{d}F$ denotes the {differential} of $F$ and we use $F''$ to denote the Hessian ({the matrix of second derivatives}). The $2\times2$ {Hessian} with respect to $s$ and $t$ is seen to be 
\be\label{Jacob sngl}
\kbordermatrix{ & \partial_t & \partial_s \\
F(y(s) + t\Theta) & 0 & \dd F(\vx_0) y^\prime(s_0) \\
\mathrm{d}F(y(s) + t\Theta)\Theta & \Theta_0^{\top}F^{\prime\prime}(\vx_0)\Theta_0 & *},
\ee
where the value of $*$ is not needed here as we are interested in the determinant. The matrix in \eqref{Jacob sngl} is evaluated at $(s,\Theta,t) = (s_0,\Theta_0,t_0)$, where $\vx_0 = y(s_0) + t_0\Theta_0$. The rows and columns are labeled accordingly depending on the term in \eqref{tang cond} considered and the derivative that is being taken. The value of zero in the upper left corner occurs because 
\be
\pa_t F(y(s_0)+t\Theta_0)|_{t=t_0}=\dd F(\vx_0)\Theta_0=0.
\ee
This implies that $s$ is a smooth function of $\Theta$ if (i) the curvature of $\s$ in the direction of $\Theta_0$ is non-zero, i.e., $\Theta_0^{\top}F^{\prime\prime}(\vx_0)\Theta_0\not=0$, and (ii) $L_{\vy_0}$ is not a critical line, i.e., $\dd F(\vx_0) y^\prime(s_0)\not=0$. Conditions (i) and (ii) follow immediately from Assumption \ref{ass:sngl}, (1) and (2), respectively.

To prove the last claim of the lemma, we compute a vector $\eta\in N_{\vy}^*(\Gamma)\setminus\{0\}$, where $\vy\in B_{\vy_0}(\de)$, and show that $(\vy,\eta)\in\cC(\vx,\xi)$, where $\vx=y(s)+t\Theta$ and $\xi\in N_{\vx}^*\s$. Denote $H_1(\vy,t) = F(y(s)+t\Theta)$ and $H_2(\vy,t) = \dd F(y(s)+t\Theta)\Theta$. These are the first and second functions in \eqref{tang cond}, respectively. To find $\eta$ we solve $H_2(\vy,t)=0$ for $t$ in terms of $\vy$ and substitute into $H_1$. Then
\be\label{eta tang}
\eta=\dd_{\vy}H_1-\frac{\pa_t H_1}{\pa_t H_2}\dd_{\vy} H_2=\dd F(y^\prime\ t\Theta_\al^\top\ t\Theta_\bt^\top),
\ee
where the last factor is a $3\times 3$ matrix. We have used here that $\pa_t H_1=\dd F\Theta=0$
and $\pa_t H_2=\Theta^\top F^{\prime\prime}\Theta\not=0$. Since $\dd F=\la\xi$ for some $\la\neq0$, \eqref{eta tang} implies
\be\label{eta tang v2}
\eta=\la\xi(y^\prime\ t\Theta_\al^\top\ t\Theta_\bt^\top).
\ee
The component $\eta_s$ matches the formula in \eqref{CR alt} because $\xi=\vsig\hat\Theta^\prime$. The fact that the component $\eta_\theta$ matches the formula in \eqref{CR alt} follows from the identity
\be
t\vsig\hat\Theta^{\prime\prime}=\vsig\hat\Theta^\prime(t\Theta_\al^\top\ t\Theta_\bt^\top)
\ee
The uniqueness of $(\vx,\xi)$ given $(\vy,\eta)$ follows from Lemma~\ref{lem:bolk fail} because we assumed that $\Pi(\vx_0,\xi)$ is not tangent to $\ga$ at $y(s_0)$. In the other direction, it is clear from the above argument that $(\vy,\eta)=\cC(\vx,\xi)$ is unique given $(\vx,\xi)$.
\end{proof}

\subsection{$\mathcal{R}\mu$ as a conormal distribution}\label{ssec:conorm}
Recall that $B_{\vy_0}(\delta)$ is a small neighborhood of $\vy_0$.

\begin{lemma}\label{lem:sing-ty} Suppose Assumption~\ref{ass:sngl} holds and $\vx_0$ is the only point where $L_{\vy_0}$ is tangent to $\s$. One has 
\be
\label{dist_R}
R\mu_E(\vy) =  g_1(\vy,E) + g_2(\vy,E)(s - \phi(\Theta))^{1/2}_\imath ,\ \vy\in B_{\vy_0}(\delta),
\ee
where $g_1\geq 0$ and $g_2$ are smooth; $g_2(\vy,E) \neq 0$ for $\vy\in\Gamma$; $\imath = +$ or $-$ depending on the orientation of $\s$ relative to $L_{\vy_0}$; $s = \phi(\Theta)$ is the local equation of $\Gamma$ near $\vy_0$, and $\phi$ is smooth.  
\end{lemma}

See \cite{rz1, rz5, airapramm01} for closely related results.

\begin{proof} By assumption, $\Pi(\vx_0,\xi_0)$ is not tangent to $\ga$ at $y(s_0)$, where $\xi_0=\dd F(\vx_0)\in N_{\vx_0}^*\s$ and $F(\vx)=0$ is a local equation of $\s$. By construction, $\vx_0=y(s_0)+t_0\Theta_0$ for some $t_0\not=0$. Thus,
\be
F(\vx_0)=0,\quad \pa_s F(y(s)+t_0\Theta_0)|_{s=s_0}=\dd F(\vx_0)y^\prime(s_0)\not=0. 
\ee
This implies that $F(y(s)+t_0\Theta_0)$ changes sign as $s$ passes through $s_0$. Furthermore, since
\be\bs
\pa_t F(y(s_0)+t\Theta_0)|_{t=t_0}=&\dd F(\vx_0)\Theta_0=0,\\ 
\pa_t^2 F(y(s_0)+t\Theta_0)|_{t=t_0}=&(F^{\prime\prime}(\vx_0)\Theta_0,\Theta_0)\not=0, 
\end{split}
\ee
we conclude that the lines $\{y(s)+t\Theta_0,t\in\br\}$, $|s-s_0|\ll1$, intersect $\s$ twice on one side of $s_0$ and do not intersect $\s$ at all on the other side of $s_0$. If, for example, $(F^{\prime\prime}(\vx_0)\Theta_0,\Theta_0)>0$, then the intersections exist on the side of $s_0$ where $F(y(s)+t_0\Theta_0)<0$.

Suppose the intersections exist for $s>s_0$. Let $t_1(s)<t_2(s)$ be the two solutions to the equation $F(y(s)+t\Theta_0)=0$ when $s_0<s<s_0+\kappa$ for some $0<\kappa\ll1$. Using Morse's lemma, we obtain that $t_2(s)-t_1(s)=g_2(s)(s-s_0)^{1/2}$ for some smooth $g_2$ that satisfies $g_2(s_0)\not=0$. This proves the result when $\Theta=\Theta_0$. Using Lemma~\ref{lem:Gamma_sm} and Morse's lemma with parameters \cite[Lemma C.6.1]{hor3}, we extend this representation to a neighborhood of $\Theta_0$ to obtain \eqref{dist_R}.

The function $g_1(\vy,E)$ arises from integrating the function $\sum_{i \not= i_0} \Phi_i(\vx,E) \chi_{\Omega_i}(\vx)$ (see \eqref{mu_model}), where $i_0$ is such that $\vx_0\in\pa\Omega_{i_0}$. The function $g_1$ is smooth, because we assumed that all other intersections $L_{\vy_0}\cap(\s\setminus \pa\Omega_{i_0})$ are transversal.
\end{proof}

We now have the theorem which addresses how the non-linearities in $\mathcal{R}$ affect such singularities.
\begin{theorem}
\label{log_sing}
Let $R\mu_E$ take the form \eqref{dist_R} locally near $\vy_0$. Then, for $\delta$ small enough, we have
\be
\label{sqrt_sing}
\mathcal{R}\mu(\vy) = h_0(\vy) + h_1(\vy) (s - \phi(\Theta))^{1/2}_\imath + h_2(\vy) (s - \phi(\Theta))_\imath,\ \vy\in B_{\vy_0}(\delta),
\ee
where $h_i$, $i=0,1,2$, are smooth. Further,
\be\label{Rmu distr}
\CR\mu \in \dot{I}^{-3/2}(\Gamma)=I^{-7/4}(N^*\Gamma)
\ee
locally near $\vy_0$.
\end{theorem}
\begin{proof}


Introduce the function
\be
H(\vy,p) = -\log\Big(\int_0^{E_{\mathrm{mx}}}P(E) e^{-(g_1(\vy,E) + g_2(\vy,E)p)} \dd E\Big),\ \vy\in B_{\vy_0}(\delta),p\in\br.
\ee
Then
\be
\label{nlR 1}
\CR\mu(\vy) = H(\vy,(s - \phi(\Theta))^{1/2}_\imath),\ \vy\in B_{\vy_0}(\delta).
\ee
Since, {by our physical assumptions,} $P(E) \geq 0$ is integrable {and $\int_0^{E_{\text{mx}}}P(E)\mathrm{d}E > 0$,} it is clear that $H$ is a smooth function of $\vy$ and $p$.  

Consider the even and odd parts of $H(p)$:
\be\label{tH 1st}
H_e(p):=[H(p)+H(-p)]/2,\quad
H_o(p):=[H(p)-H(-p)]/2,
\ee
where the dependence of $H$ on $\vy$ is temporarily suppressed for simplicity. Introduce
\be\label{tq defs 1st}
h_1(p):=H_o(p)/p,\quad h_2(p):=(H_e(p)-H(0))/p^2.
\ee
Clearly, $h_1(p)$ and $h_2(p)$ are smooth and even. The $h_i$ are still smooth at zero since the leading order terms of the Taylor expansion of $H_o$ and $H_e$ are $p$ and $p^2$, respectively. Let $h_i(\vy)$, $i=1,2$, be the functions obtained by substituting $p=|s - \phi(\Theta)|$ into \eqref{tq defs 1st}. By construction, $h_i(\vy)$ are smooth in $B_{\vy_0}(\de)$. Moreover, \eqref{nlR 1}--\eqref{tq defs 1st} imply
\be
\label{nlR 2}
\CR\mu(\vy) = H(\vy,0)+h_1(\vy)(s - \phi(\Theta))^{1/2}_\imath + h_2(\vy)(s - \phi(\Theta))_\imath,\ \vy\in B_{\vy_0}(\delta),
\ee
which proves \eqref{sqrt_sing} with $h_0(\vy)=H(\vy,0)$.

Now, by Proposition \ref{prop_sqrt}, we have 
\be
\label{dist_order}
g(\vy) (s - \phi(\Theta))^r_\imath \in \dot{I}^{-(r+1)}(\Gamma)=I^{-(r+5/4)}(N^*\Gamma)
\ee
locally near $\vy_0$, for any smooth $g$. The second claim of the theorem follows from \eqref{sqrt_sing} after setting $r = 1/2$ in \eqref{dist_order}.
\end{proof}

In our next theorem, we quantify the Sobolev order of the singularities of $\mathcal{R}\mu$, and how they relate to the singularities of $Rf$. 

\begin{theorem}\label{WF_thm}
Suppose Assumption~\ref{ass:sngl} holds and $\vx_0$ is the only point where $L_{\vy_0}$ is tangent to $\s$. Let $(\vy_0,\eta_0)\in N^*\Gamma$. We have
\be
\WF^1(\mathcal{R}\nu) = \WF(Rf)
\ee
microlocally near $(\vy_0,\eta_0)$. 
\end{theorem}

\begin{proof}
The inclusion \eqref{Rmu distr} and Theorem~\ref{thm:embedding} imply that $\mathcal{R}\mu\in H^\nu$ near $\vy_0$ for any $\nu<1$. Furthermore, $\mathcal{R}\mu$ adopts a square root type singularity near $\vy_0$ by Corollary \ref{log_sing} and \eqref{sqrt_sing}. As such, the first derivative of \eqref{sqrt_sing} with respect to $s$ cannot be in $L^2$. Therefore, $\mathcal{R}\mu$ is not in $H^1$ locally near $\vy_0$. \end{proof}

\section{Analysis of singularties due to non-linearity - double tangent lines}\label{sec:dbl tang}
We now consider the case when $L_{\vy_0}$ is tangent to $\s$ at two distinct points, $\vx_1$ and $\vx_2$. Let $\s_k$ be a sufficiently small local segment of $\s$ near $\vx_k$, $k=1,2$. Let $\Gamma_k$ be the local segment of $\Gamma$ in a neighborhood of $\vy_0$ corresponding to $\s_k$, $k=1,2$ (cf. \eqref{Ga def}). We will use this notation throughout the rest of the paper.

In what follows we will always assume that if $L_{\vy_0}$ is a double tangent, then no plane $\Pi(\vx_k,\xi_k)$, $k=1,2$, is tangent to $\ga$, where $\xi_k\in N_{\vx_k}^*\s$. This assumption is generically true, as is illustrated by the following example.\vspace{2mm}

\noindent
{\bf Example.} Let $f = \chi_{B_1} + \chi_{B_2}$, where $B_1,B_2 \subset X$ are disjoint balls. 
Let $\Sigma_i\subset X$ be the artifacts in $R^*R \chi_{B_i}$ for $i = 1,2$. 

Generically, $\Sigma_i$ is the union of several co-dimension 1 smooth surfaces. Indeed, for each $y(s)\in\ga$, there are at most two planes that are tangent to $\ga$ at $y(s)$ and to $B_i$. Each such plane contributes a line $L\in \Sigma_i$. The line $L$ passes through $y(s)$ and the point of tangency on the ball.

The assumption holds if $\Sigma_1$ is not tangent to $\partial B_2$ and, similarly, $\Sigma_2$ is not tangent to $\partial B_1$. The set of balls tangent to $\Sigma_i$ is three dimensional, but the full set of balls in $\br^3$ is four dimensional. Hence, $\Sigma_i$ is tangent to $\partial B_j$ with probability zero. That is, we can construct specific examples where the assumption fails, but it is not likely to fail in a practical setting.\vspace{2mm}

We wish to note that, while triple tangent rays are possible in certain special cases, they do not often occur in practice, and thus we consider only single and double tangent rays in this paper.

We now state the assumption which will be needed for our analysis of double tangent singularities.

\begin{assumption}\label{ass:dbl}
Suppose $L_{\vy_0}$ is a double tangent, i.e. it is tangent to $\s$ exactly at two distinct points, $\vx_1$ and $\vx_2$. We assume that Assumption~\ref{ass:sngl} holds at each $\vx_k$, $k=1,2$.
\end{assumption}

Our first theorem describes the singular behavior of the nonlinear data near $\vy_0$. The theorem is analogous to \cite[Theorem 3.1]{katsevich2026}.

\begin{theorem}
\label{double_sing}
Suppose $L_{\vy_0}$ is tangent to $\s$ exactly at two distinct points, $\vx_1$ and $\vx_2$, and Assumption~\ref{ass:sngl} holds at each $\vx_k$, $k=1,2$. For $\delta > 0$ sufficiently small, we have
\be
\label{double_lin}
\mathcal{R}\mu(\vy) = \sum_{i=0}^2\sum_{j=0}^2q_{ij}(\vy)(s - \phi_1(\Theta))_{\imath_1}^{i/2}(s - \phi_2(\Theta))_{\imath_2}^{j/2},\ \vy\in B_{\vy_0}(\delta),
\ee
where the $q_{ij}$ are smooth and we define $x_\imath^0 \equiv 1$. 
\end{theorem}

\begin{proof}
By an easy generalization of Lemma~\ref{lem:sing-ty} we have
\be
\label{dist_R_1}
R\mu_E(\vy) =  g_0(\vy,E) + g_1(\vy,E) (s - \phi_1(\Theta))^{1/2}_{\imath_1} + g_2(\vy,E) (s - \phi_2(\Theta))^{1/2}_{\imath_2},\ \vy\in B_{\vy_0}(\delta),
\ee
where the $g_k$ is smooth and $s- \phi_k(\Theta) = 0$ is the defining equation for $\Gamma_k$ near $\vy_0$ for $k = 1,2$. 

Similarly to the proof of Theorem \ref{log_sing}, we have
\be
\label{equ_mid}
\begin{split}
\CR\mu(\vy) =&H\big(\vy,(s - \phi_1(\Theta))_{\imath_1}^{1/2},(s - \phi_2(\Theta))_{\imath_2}^{1/2}\big),\ \vy\in B_{\vy_0}(\delta),\\
H(\vy,p,q) :=&-\log\Big(\int_0^{E_{\mathrm{mx}}}P(E) e^{-(g_0(\vy,E) + g_1(\vy,E)p+ g_2(\vy,E)q)} \dd E\Big),\ p,q\in\br.
\end{split}
\ee 

Introduce the following four functions
\be\label{tH fns}\bs
H_{e,e}(p,q):=&[H(p,q)+H(-p,q)+H(p,-q)+H(-p,-q)]/4,\\
H_{o,e}(p,q):=&[H(p,q)-H(-p,q)+H(p,-q)-H(-p,-q)]/4,\\
H_{e,o}(p,q):=&[H(p,q)+H(-p,q)-H(p,-q)-H(-p,-q)]/4,\\
H_{o,o}(p,q):=&[H(p,q)-H(-p,q)-H(p,-q)+H(-p,-q)]/4.
\end{split}
\ee
The first and second subscripts denote the parity of the corresponding function with respect to $p$ and $q$, respectively. Next, we introduce
\be\label{tq defs}\bs
q_{22}(p,q):=&[H_{e,e}(p,q)-H_{e,e}(p,0)-H_{e,e}(0,q)+H(0,0)]/(p^2q^2),\\
q_{12}(p,q):=&[H_{o,e}(p,q)-H_{o,e}(p,0)]/(pq^2),\\
q_{21}(p,q):=&[H_{e,o}(p,q)-H_{e,o}(0,q)]/(p^2q),\\
q_{11}(p,q):=&H_{o,o}(p,q)/(pq),\\
q_{20}(p,q):=&(H_{e,e}(p,0)-H(0,0))/p^2,\quad q_{10}(p,q):=H_{o,e}(p,0)/p,\\
q_{02}(p,q):=&(H_{e,e}(0,q)-H(0,0))/q^2,\quad q_{01}(p,q):=H_{e,o}(0,q)/q.
\end{split}
\ee
Define also $q_{00}(p,q)\equiv H(0,0)$. The functions $q_{i0}(p,q)$ and $q_{0i}(p,q)$, $i=1,2$, are independent of $q$ and $p$, respectively. We retain both variables in the arguments for notational simplicity. Let $q_{ij}(\vy)$ be the functions obtained by substituting $p=|s - \phi_1(\Theta)|^{1/2}$ and $q=|s - \phi_2(\Theta)|^{1/2}$ into \eqref{tq defs}. Clearly, $q_{ij}(\vy)$ are smooth in $B_{\vy_0}(\de)$. Combining \eqref{equ_mid}--\eqref{tq defs} proves \eqref{double_lin}.
\end{proof}

Next, we study the geometry of the intersection $\Gamma_1\cap\Gamma_2$.

\begin{theorem}
\label{trans_int} Let $L_{\vy_0}$ be tangent to $\s$ exactly at two points, $\vx_1$ and $\vx_2$.  Suppose Assumption \ref{ass:sngl} holds at $\vx_k$, $k=1,2$. Then, locally, $\Gamma=\Gamma_1\cup\Gamma_2$, where $\Gamma_k$, $k=1,2$, are smooth. Moreover, $\vy_0\in\Gamma_1\cap\Gamma_2$ and the intersection is transverse. 
\end{theorem}

\begin{proof} By Assumption \ref{ass:sngl} and Lemma~\ref{lem:Gamma_sm}, $\Gamma_1$ and $\Gamma_2$ are smooth in a neighborhood of $\vy_0$. Suppose $\Gamma_1$ and $\Gamma_2$ intersect non-transversally at $\vy$. Then the conormal vectors to $\Gamma_1$ and $\Gamma_2$ at $\vy_0$ are both parallel, say, to some $\eta_0$. By the last claim of Lemma~\ref{lem:Gamma_sm}, there exist $\xi_k \in N_{\vx_k}^*\s$, $k=1,2$, such that $\mathcal{C}(\vx_1,\xi_1) = \mathcal{C}(\vx_2,\xi_2) = (\vy_0,\eta_0)$.  By Lemma~\ref{lem:bolk fail}, this is possible only if $\xi_1 y^\prime=\xi_2y^\prime=0$. However, this is ruled out by Assumption~\ref{ass:sngl}\eqref{sng transv} applied at $\vx_1$ and $\vx_2$.
\end{proof}

By the last theorem, $\Gamma_1\cap\Gamma_2$ is a 1-D smooth curve, and each $\vy\in \Gamma_1\cap\Gamma_2$ defines a double tangent line. 
Together, the family of double tangent lines form the surface
\be\label{art surface}
\CL_a:=\cup_{\vy\in \Gamma_1\cap\Gamma_2} (L_{\vy}\cap X).
\ee
Later in section \ref{sec: sing recon}, we will show that $\mathcal{L}_a$ describes the locations of the artifacts in the reconstruction due to beam hardening. Hence, here and below the subscript $a$ stands for the artifact. 

By the above analyses, we can now see how the non-linearities in the data can create new singularities on $\Gamma_1 \cap \Gamma_2$ due to the product terms in \eqref{double_lin}, and the transverse intersection property proven in Theorem \ref{trans_int}. We begin to formalize this idea in the following lemma, {where we also} quantify the strength of the added singularities. Throughout the rest of the paper, we use the notation $\Xi_a = N^*(\Gamma_1\cap\Gamma_2)$. That is, $\Xi_a$ describes a local piece of $N^*(\Gamma_1\cap\Gamma_2)$ near some $\vy_0$.

\begin{lemma}
\label{lemma_prod} Suppose Assumption~\ref{ass:dbl} holds. Let $i_1,i_2 \geq 1$. We have
\be
\begin{split}
&(s - \phi_1(\Theta))_{\imath_1}^{i_1/2}(s - \phi_2(\Theta))_{\imath_2}^{i_2/2}\\
 &\in I^{-\frac{i_1}{2} - \frac{5}{4},-\frac{i_2+1}{2}}(\Xi_a, N^*\Gamma_1) + I^{-\frac{i_2}{2} - \frac{5}{4},-\frac{i_1+1}{2}}(\Xi_a, N^*\Gamma_2)
\end{split}
\ee
locally near $\vy_0$.
\end{lemma}
\begin{proof}
By Proposition \ref{prop_sqrt}, we have 
\be
u_k = (s - \phi_k(\Theta))_{\imath_k}^{i_k/2}\in \dot{I}^{\nu_k}(\Gamma)=I^{\nu_k-1/4}(N^*\Gamma),\ k=1,2,
\ee 
near $\vy_0$, where $\nu_k = -(i_k/2)-1$. We aim to use Lemma \ref{lem_green} to prove the result. First, we have $\text{codim}_X(\Gamma_{1,2}) = d_1 = 1$ and $\text{codim}_X(\Gamma_1 \cap\Gamma_2) = d_1 + d_2 = 2 \implies d_2 = 1$, since, by Theorem \ref{trans_int}, the intersection of $\Gamma_{1}$ and $\Gamma_2$ is transverse near $\vy_0$. The application of \eqref{two types} and Lemma \ref{lem_green} yields the result 
\be\label{u1u2}
\begin{split}
u_1u_2 &\in I^{\nu_1,\nu_2}(\Gamma_1,\Gamma_1\cap\Gamma_2) + I^{\nu_2,\nu_1}(\Gamma_2,\Gamma_1\cap\Gamma_2) \\
&=  I^{\nu_1 - \frac{1}{4},\nu_2+\frac{1}{2}}(\Xi_a, N^*\Gamma_1) + I^{\nu_2 - \frac{1}{4},\nu_1+\frac{1}{2}}(\Xi_a, N^*\Gamma_2),
\end{split}
\ee
after inputting the values for $d_{1,2}$ as above. Substituting the values of $\nu_{1,2}$ proves the lemma.
\end{proof}

Let $u_k = (s - \phi_k(\Theta))_{\imath_k}^{i_k/2}$, $k=1,2$, as in the proof above. From  Definition~\ref{pair_lag} and \eqref{u1u2}, we have
\be\label{wf nl}
\WF_{\vy_0}(\mathcal{R}\mu) = \WF_{\vy_0}(u_1u_2) \subset  N_{\vy_0}^*(\Gamma_1\cap\Gamma_2)=\text{span}(\eta_1,\eta_2),
\ee
where $\eta_k$ are conormal to $\Gamma_k$ at $\vy_0$ for $k = 1,2$. This also follows using the results of Hormander \cite[Theorem 8.2.10]{hor}. This is in contrast to the linear analog, where
\be\label{wf lin}
\WF_{\vy_0}(Rf) = \text{span}(\eta_1) \cup \text{span}(\eta_2).
\ee
Thus, the data non-linearity creates new singularities at points $\vy_0$ corresponding to double tangent lines. At each such $\vy_0$, the added conormal directions lie in the plane spanned by $\eta_k\in N_{\vy_0}^*(\Gamma_k)$, $k=1,2$. 

By Theorem \ref{double_sing}, $\mathcal{R}\mu$ takes the form
\be
\label{equ_sum}
\mathcal{R}\mu(\vy) = \mathcal{T}(\vy) + \sum_{i=1}^2\sum_{j=1}^2q_{ij}(\vy)(s - \phi_1(\Theta))_\imath^{i/2}(s - \phi_2(\Theta))_\imath^{j/2} = \mathcal{T}(\vy) + \mathcal{N}(\vy)
\ee
near $\vy_0$. Here, $\WF(\mathcal{T}) = \WF(Rf)$ near $\vy_0$, and $\mathcal{T}$ behaves microlocally like the linear transform. Further, by \eqref{Rmu distr},
\be\label{cal_L}
\mathcal{T} = u_1' + u_2',\ u_k' \in I^{-7/4}(N^*\Gamma_k),\ k=1,2. 
\ee
The term $\mathcal{N}$ arises from the non-linearities and 
\be
\label{cal_N}
\mathcal{N}\in I^{-\frac{7}{4},-1}(\Xi_a, N^*\Gamma_1) + I^{-\frac{7}{4},-1}(\Xi_a, N^*\Gamma_2)
\ee
by Lemma \ref{lemma_prod}. This follows after setting $i_1 = i_2 = 1$ in Lemma \ref{lemma_prod} and noting that the lower order terms of $\mathcal{N}$ (i.e., in the sum in \eqref{equ_sum}) have stronger (lower Sobolev order) singularities.

We now quantify the Sobolev order of the added singularities of $\mathcal{R}\mu$, and how they compare to the singularities of $Rf$ in the case when $L_{\vy_0}$ is a double tangent. Pick any $(\vy_0,\eta)\in N^*(\Gamma_1\cap\Gamma_2)\setminus(N^*\Gamma_1\cup N^*\Gamma_2)$. Thus, by \eqref{wf nl} and \eqref{wf lin},
\be
\eta \in \text{span}(\eta_1,\eta_2)\backslash \big(\text{span}(\eta_1) \cup \text{span}(\eta_2)\big),
\ee
where the $\eta_k$ are the conormals to $\Gamma_k$ at $\vy_0$. By \eqref{equ_sum}, $\mathcal{R}\mu = \mathcal{T} + \mathcal{N}$ near $\vy_0$. The equations \eqref{cal_L} and \eqref{cal_N} imply that $(\vy_0,\eta)\notin \WF(\mathcal{T})$.
By \eqref{cal_N} and Definition \ref{pair_lag}, we have
\be
\mathcal{N} \in I^{-\frac{11}{4}}\big( \Xi_a \backslash (N^*\Gamma_1\cup N^*\Gamma_2)\big),
\ee
near $\vy_0$. Theorem~\ref{thm:embedding} implies that $\mathcal{N}\in H^{\alpha}$ for any $\alpha < 2$, locally near $(\vy_0, \eta)$.

Thus, the additional singularities in the data created by the product terms in $\mathcal{N}$ are smoother than the singularities which correspond to $\WF(Rf)$.

\section{Properties of the composition $\cC^t\circ\Xi_a$}\label{sec:props comp}

We now aim to describe how the added singularities in the data, along $\Xi_a$, affect the reconstruction. As in Theorem \ref{lin_art}, we consider the reconstruction operator $R^* \mathcal{P}$, where $\mathcal{P}$ is a $\Psi$DO. This has the form of FBP, where $\mathcal{P}$ represents the filter. To assess the resulting artifacts, it is important to understand how the canonical relation of $R^*$, namely $\Cc^t$, propagates the new singularities in $\Xi_a$ back to the image domain. This is the purpose of this section.

In our next theorem, we describe the location and direction of the artifacts in the reconstruction caused by the additional singularities in $\mathcal{N}$. Recall that $\CL_a$ is defined in \eqref{art surface}.

\begin{theorem}\label{thm:C nonlin_art}
Suppose $L_{\vy_0}$ is a double tangent line and  Assumption~\ref{ass:dbl} holds. One has
\be
\label{nonlin_art_equ}
\cC^t\big(\Xi_a\big)\subset  N^*\CL_a.
\ee
Moreover, if $(\vy,\eta)\in \Xi_a$ and $(\vx_a,\xi_a)=\cC^t(\vy,\eta)$ (which, in particular, implies $\vx_a\in X$), then
\be
\eta_s=\xi_a\, y^\prime(s)\not=0.
\ee
\end{theorem}

{\begin{remark}
    Geometrically, the second claim of the theorem means that the tangent plane to $\mathcal{L}_a$ at $\vx_a$ (i.e., $\Pi(\vx_a,\xi_a)$, using the notation of the theorem) intersects $\gamma$ transversally at $y(s)$.  This plane may be tangent to $\ga$ at other points, however.
\end{remark}}

\begin{proof} 
Define $\xi_k(s):=\dd F_k(\vx_k(s))\in N_{\vx_k(s)}^*(\s_k)$, where $\vx_k(s)=y(s)+t_k(s)\Theta_a(s)$ is the point where $L_{\vy_a(s)}$ is tangent to $\s_k$, $k=1,2$, and $s$ is close to $s_0$. By the last claim of Lemma~\ref{lem:Gamma_sm}, we can find $\eta^{(k)}(s)\in N_{\vy_a(s)}^*(\Gamma_k)$ so that $(\vy_a(s),\eta^{(k)}(s))=\mathcal{C}(\vx_k(s),\xi_k(s))$, $k=1,2$. To simplify the notation, in what follows we omit the dependence of most quantities on $s$. 

Let $\eta=(\eta_s,\eta_\theta)$, $\eta_{\theta}=(\eta_\al,\eta_\bt)$. We can rewrite \eqref{eta-xi} in the form
\be\label{eta-xi-new}
\eta_s=\vsig\hat\Theta^\prime y^\prime,\ \eta_\theta=t\vsig\hat\Theta^{\prime\prime},\ 
\xi=\vsig\hat\Theta^\prime.
\ee
From \eqref{useful matr}, \eqref{CR alt}, and \eqref{eta-xi-new}, we have
\be\label{eta(xi)}
\eta_s^{(k)}=\xi_k y^\prime,\ \eta_\theta^{(k)}=t_k\xi_k(\hat\Theta^\prime)^{-1}\hat\Theta^{\prime\prime},\ k=1,2.
\ee

To describe the artifact microlocally we need to solve
\be\label{etas}
\eta_s=\vsig\hat\Theta^\prime y^\prime,\ \eta_\theta=t\vsig\hat\Theta^{\prime\prime},\ \eta:=c_1\eta^{(1)}+c_2\eta^{(2)}\in\mathrm{span}(\eta^{(1)},\eta^{(2)}),
\ee
for $t$ and $\vsig$. Letting $\eta^{(k)} = \begin{pmatrix}\eta_s^{(k)} & \eta_\theta^{(k)}\end{pmatrix}$, the solution is found to be
\be\label{ta}
t_a=\frac{  [c_1\eta_\theta^{(1)}+c_2\eta_\theta^{(2)}](\hat\Theta^{\prime\prime})^{-1} \hat\Theta^\prime y^\prime}{c_1\eta_s^{(1)}+c_2\eta_s^{(2)}},\ \vsig_a=\frac 1{t_a} [c_1\eta_\theta^{(1)}+c_2\eta_\theta^{(2)}](\hat\Theta^{\prime\prime})^{-1}.
\ee
The corresponding artifact point $(\vx_a,\xi_a)=\cC^t(\vy,\eta)$ is then given by 
\be\label{xia}
\vx_a=y(s)+t_a\Theta_a,\quad \xi_a=\vsig_a\hat\Theta^\prime=(1/t_a)[c_1\eta_\theta^{(1)}+c_2\eta_\theta^{(2)}] (\hat\Theta^{\prime\prime})^{-1}\hat\Theta^\prime.
\ee
Substitute \eqref{eta(xi)} into \eqref{ta} and \eqref{xia}:
\be\label{ta xia}
t_a=\frac{c_1t_1(\xi_1 y^\prime)+c_2t_2(\xi_2 y^\prime)}{c_1(\xi_1 y^\prime)+c_2(\xi_2 y^\prime)},\ \xi_a=\frac 1{t_a} (c_1t_1\xi_1+c_2t_2\xi_2).
\ee
By Assumption~\ref{ass:dbl}, $\xi_k y^\prime\not=0$, $k=1,2$. 
Since $t_1\not=t_2$ (the two points of tangency $\vx_k(s)$ are distinct), the first equation in \eqref{ta xia} implies that any $t_a\in\br$ can, in principle, be obtained by selecting $c_1$ and $c_2$ appropriately. However, due to the restriction $\vx_a\in X$, not all values of $t_a$ occur. From \eqref{ta xia} we see that there is only one covector (up to scaling) corresponding to each base point on the artifact surface $\CL_a$. 

The conditions $t_1\not=t_2$, $\xi_k y^\prime\not=0$, $k=1,2$, and $\vx_a\in X$, also imply that in the formula for $t_a$, the numerator and denominator cannot simultaneously be zero. This means that the denominator is bounded away from zero. Otherwise, $t_a\to\infty$, and the corresponding artifact point is outside $X$. The numerator cannot be zero either, because $X$ and $\ga$ are disjoint by assumption.

The second equation in \eqref{ta xia} implies that $\xi_a\Theta_a=0$ because, by construction, $\xi_k\Theta_a=0$, $k=1,2$. To show that $\xi_a$ is (co)normal to $\CL_a$ it remains to prove that 
\be\label{scnd cond}
\xi_a \pa_{\mu}(y(s)+t_a\Theta_a)=\xi_a \pa_{\mu}(y^\prime(s)\pa_\mu s+t_a\pa_\mu\Theta_a)=0,
\ee
where $s$, $t_a$, and $\Theta_a$ are functions of $\mu$. Recall that $\mu$ is the parameter along $\Gamma_1\cap\Gamma_2$. 

By construction, $F_k(y(s(\mu))+t_k(\mu)\Theta_a(\mu))\equiv0$. Recall that $F_k(x)=0$ is a local equation of $\s_k$. Differentiating with respect to $\mu$ and using the double tangency we find
\be\label{xii norm}
\xi_k\big(y^\prime(s)\pa_\mu s+t_k\pa_\mu\Theta_a\big)=0,\ k=1,2.
\ee
Substituting \eqref{xii norm} and \eqref{ta xia} into \eqref{scnd cond} and dropping the $1/t_a$ prefactor we compute
\be
(c_1t_1\xi_1+c_2t_2\xi_2) (y^\prime(s)\pa_\mu s+t_a\pa_\mu\Theta_a^\prime)
=\big[c_1(\xi_1 y^\prime)(t_1-t_a)+c_2(\xi_2 y^\prime)(t_2-t_a)\big]\pa_\mu s=0.
\ee
We have now proven that the nonlinear artifacts are in $N^*\CL_a$.

To prove the second claim, suppose the artifact point $\vx_a\in X$ is such that $\xi_a\, y^\prime=0$. It follows from the last equation in \eqref{eta-xi-new} and the first equation in \eqref{etas} that $\eta_s=\xi_a\, y^\prime$. Therefore the statements $\eta_s=0$ and $\xi_a\, y^\prime=0$ are equivalent provided that $\cC(\vx_a,\xi_a)=(\vy,\eta)$. As mentioned above, $\vx_a\not\in\ga$ implies $t_a\not=0$. The second equation in \eqref{ta xia} implies
\be\label{ta xia cor}\begin{split}
0=t_a(\xi_a y^\prime)=&c_1t_1(\xi_1 y^\prime)+c_2t_2(\xi_2 y^\prime)\\
=&c_1t_1 \eta_s^{(1)}+c_2t_2\eta_s^{(2)}.
\end{split}
\ee
Together with the first equation in \eqref{ta xia} this leads to the contradiction $t_a=0$.   
\end{proof}

The proof of the preceding theorem explicitly describes the mechanism of how the application of $\cC^t$ to the singularities of the data maps them into the image domain. The following lemma derives a somewhat similar result. It is more powerful than the former, but its proof is less revealing.

\begin{lemma}\label{lem_clean}
Suppose $L_{\vy_0}$ is a double tangent line and Assumption~\ref{ass:dbl} holds. 
The composition $\Cc^t \circ \Xi_a$ is clean with excess zero.
\end{lemma}

\begin{proof}
Denote
\be
L=\cC^t\times\Xi_a,\ M=(T^*X\setminus \{0\})\times \Lambda_Y^*,
\ee
where $\cC^t$ is the canonical relation of $R^*$, and $\Lambda_Y^*$ is the diagonal in $(T^*Y\setminus \{0\})\times (T^*Y^\prime\setminus \{0\})$, where $Y^\prime$ is a copy of $Y$. Clearly,
\be\label{LcapM}
L\cap M=\big\{(\vx,\xi,\vy,\eta,\vy,\eta) : (\vx,\xi;\vy,\eta) \in \mathcal{C}^t, (\vy,\eta) \in \Xi_a\big\}.
\ee

According to the second assertion of Theorem~\ref{thm:C nonlin_art}, if $\omega=(\vx,\xi,\vy,\eta,\vy,\eta)\in L\cap M$, then $\eta_s=\xi\cdot y^\prime\not=0$. Together with Lemma~\ref{lem:bolk fail} this implies that $\cC$ and, therefore, $\cC^t$, is the local canonical graph.


To prove clean intersection (cf. Definition~\ref{def:clean int}), we first establish that $L\cap M$ is a manifold. It is clear that $\Xi_a$ is a smooth, three-dimensional manifold: one parameter (say, $\mu$) runs along the curve $\Gamma_1\cap\Gamma_2$, and two more parameters (say, $c_1,c_2$) are along the fibers. The desired result follows because $\cC$ is the local canonical graph.
Therefore, \eqref{LcapM} becomes
\be\label{LcapM1}
L\cap M=\big\{(\vx(v),\xi(v),\vy(v),\eta(v),\vy(v),\eta(v)) : v \in I_{\mathrm{loc}} \times \br^2\big\},\ v=(\mu,c_1,c_2),
\ee
where $I_{\mathrm{loc}}$ is some interval.

Next we show that $T(L\cap M)=TL\cap TM$. Pick any 
\be
\omega=(\vx,\xi,\vy,\eta,\vy,\eta)\in L\cap M.
\ee
Clearly, 
\be\bs
T_\omega L =& T_{(\vx,\xi,\vy,\eta)} \cC^t \times T_{(\vy,\eta)}\Xi_a,\\
T_\omega M=&\{(\dot\vx,\dot\xi,\dot\vy,\dot\eta,\dot\vy',\dot\eta'):\dot\vy=\dot\vy',\dot\eta=\dot\eta'\},
\end{split}
\ee
where the dots denote dual variables. Therefore,
\be\bs\label{TLcapTM}
T_\omega L\cap T_\omega M =& \{(\dot\vx,\dot\xi,\dot\vy,\dot\eta,\dot\vy,\dot\eta):(\dot\vx,\dot\xi,\dot\vy,\dot\eta)\in T_{(\vx,\xi,\vy,\eta)}\cC^t, (\dot\vy,\dot\eta)\in T_{(\vy,\eta)}\Xi_a\}.
\end{split}
\ee
The intersection in \eqref{TLcapTM} can be described as follows. Parameterize points $(\vy,\eta)\in \Xi_a$ using $v = (\mu,c_1,c_2)$ as above. Let $D$ be a linear combination of the partial derivatives $\pa_\mu$, $\pa_{c_1}$, and $\pa_{c_2}$. Take $(\dot\vy,\dot\eta)=D(\vy(v),\eta(v))$. Then 
\be
(\dot\vx,\dot\xi)=D(\vx(v),\xi(v)),\ (\vx(v),\xi(v))=\cC^t(\vy(v),\eta(v)).
\ee 
Therefore, $T_\omega L\cap T_\omega M = \text{span}(\pa_\mu \omega(v), \pa_{c_1} \omega(v),\pa_{c_2} \omega(v))$, where 
\be
\omega(v) = (\vx(v),\xi(v),\vy(v),\eta(v),\vy(v),\eta(v)), 
\ee
and, by \eqref{LcapM1}, this is precisely the set $T_\omega(L\cap M)$.

Consider the natural projection $L\cap M \to \cC^t \circ \Xi_a$. If we take $(\vx,\xi) \in N^*\CL_a = \cC^t\circ \Xi_a$, then its preimage in $L\cap M$ is the point $\omega=\big((\vx,\xi),(\vy,\eta),(\vy,\eta)\big)$, where $(\vy,\eta)=\cC(\vx,\xi)$. Since $\vy$ is near $\vy_0$, this point is locally unique. This can be easily seen from \eqref{CR alt}. Indeed, clearly there is only one point in the intersection $y(s)\in\Pi(\vx,\xi)\cap\ga$ near $\vy_0$. Then $t$ and $\Theta$ are determined uniquely. By construction, the matrix $\hat\Theta^{\prime}$ is nondegenerate (cf. \eqref{useful matr}), so we can find $\vsig$ and then $\eta$.

Likewise, the preimage of any compact subset of $\cC^t \circ \Xi_a$ under the natural projection is compact, and the map is proper. This proves that the composition $\Cc^t \circ \Xi_a$ is clean (cf. Definition~\ref{def:clean comp}).

We next use the relation \cite[page 458]{trev2} to calculate the excess $d$:
\be
\text{dim}(\Cc^t \circ \Xi_a) = \text{dim}(L\cap M) - d.
\ee
Since, by the above, $\text{dim}(N^*\CL_a) = \text{dim}(\Cc^t \circ \Xi_a) = \text{dim}(L\cap M) = 3$, this implies $d = 0$. This completes the proof.
\end{proof}

\section{Singularities in the reconstruction}\label{sec: sing recon}

We are now in a position to describe the singularities in FBP type reconstruction and the artifacts due to beam hardening.

Let $L_{\vy_0}$ be a single tangent ray and let Assumption~\ref{ass:sngl} hold. Let $\psi$ be a smooth cutoff centered on $\vy_0$. By Theorem~\ref{log_sing}, we obtain 
\be
R^* (\psi \mathcal{P} \CR f) \in I^{-\frac{9}{4} + p}(N^*\s)
\ee
for a properly supported $\Psi$DO $\CP$ of order $p$. These parts of the reconstruction behave in essentially the same way as the reconstruction from linear X-ray transform data (i.e., $R^*(\psi \mathcal{P} Rf)$), in terms of the type of singularities (conormal type) and the strength.

Consider next the double-tangent ray case. Theorems~\ref{thm:C nonlin_art} and \ref{thm:HS} imply the following result.

\begin{corollary}
\label{nonlin_art_thm}
Let $L_{\vy_0}$ be a double tangent ray as described in Section \ref{sec:dbl tang}, and let $\psi$ be a smooth cutoff centered on $\vy_0$ of sufficiently small support. Suppose Assumption~\ref{ass:dbl} holds. One has
\be
\label{nonlin_art_equ v2}
\WF(R^* (\psi\CP\mathcal{R}\mu)) \subset (\WF(f) \cap \mathcal{V}) \cup N^*\CL_a,
\ee
where $\CP$ is a properly supported $\Psi$DO.  
\end{corollary}

\begin{remark}
Since we are localizing near $\vy_0$ by multiplying by $\psi$, i.e., near a double tangent ray, the artifacts in $\Lambda$ (i.e., due to the violation of the Bolker condition) do not come into play here by Assumption~\ref{ass:dbl} and Theorem~\ref{lin_art}. This corollary specifically addresses the artifacts due to non-linearities.
\end{remark}

Next we study the strength of the added singularities.

\begin{theorem}\label{thm:precise}
Suppose $L_{\vy_0}$ is a double tangent and Assumption~\ref{ass:dbl} holds. Let $\eta_{1,2}$ be the conormals to $\Gamma_{1,2}$ at $\vy_0$, respectively, as in Theorem \ref{nonlin_art_thm} and the preceding discussions. Let $\psi$ be a smooth cutoff centered on $\vy_0$ of sufficiently small support, and let $\CP$ be a properly supported $\Psi$DO of order $p$. Then,
\be\label{R* conorm}
R^* (\psi \CP\mathcal{N}) \in I^{-\frac{13}{4} + p}(N^*\CL_a\setminus N^*\s).
\ee
\end{theorem}

It is important to multiply by $\psi$ before applying $R^*$. Otherwise, in the full recontruction there can be an added, even stronger artifact at $(\vx_a,\xi_a)$ which arises because of the violation of the Bolker condition. This may happen when $\Lambda$ and $N^*\CL_a$ intersect.


Theorem~\ref{thm:precise} provides a more powerful characterization of the added artifacts than Theorem~\ref{nonlin_art_thm}. The latter describes the added wavefront, while the former asserts that the added singularities are conormal distributions of prescribed order. 

\begin{proof}
Since $\CP$ is a properly supported $\Psi$DO of order $p$, we have by \eqref{cal_N}
\be
\mathcal{P} \mathcal{N}\in I^{-\frac74 + p,-1}(\Xi_a, N^*\Gamma_1) + I^{-\frac74 + p,-1}(\Xi_a, N^*\Gamma_2)
\ee
Therefore,
\be
\psi\mathcal{P}\mathcal{N}\in I^{-\frac{11}4 + p}\big(\Xi_a\setminus (N^*\Gamma_1 \cup N^*\Gamma_2)\big)
\ee
near $\vy_0$ microlocally away from $\eta_{1,2}$.

By Theorem~\ref{thm:C nonlin_art} and Lemma~\ref{lem_clean}, we have
\be
\Cc^t \big(\Xi_a\setminus (N^*\Gamma_1 \cup N^*\Gamma_2)\big) \subset N^*\CL_a\setminus N^*\s,
\ee
and the composition is clean by Lemma \ref{lem_clean}. By Theorem \ref{fio_thm}, $R$ and hence $R^*$ is an FIO of order $-1/2$. We now apply \cite[Chapter VIII, Theorem 5.2]{trev2} to conclude that
\be
R^*(\psi\mathcal{P}\mathcal{N})\in I^{-\frac{13}4 + p}(N^*\CL_a\setminus N^*\s),
\ee
which finishes the proof. \end{proof}

We now quantify the strength of the artifacts in $N^*\CL_a$ in the Sobolev scale. The embedding Theorem~\ref{thm:embedding} immediately implies the following result.
\begin{corollary}\label{corr_sob}
Suppose $L_{\vy_0}$ is a double tangent and Assumption~\ref{ass:dbl} holds. Let $\CL_a$ be the artifact surface defined in \eqref{art surface}. Pick any $(\vx_a,\xi_a) \in N^*\CL_a\backslash N^*\s$. Let $\CP$ be a properly supported $\Psi$DO of order $p$. Then,
\be
R^* (\psi\CP \mathcal{N}) \in H^{\alpha - p}
\ee
microlocally near $(\vx_a,\xi_a)$ for any $\alpha < 5/2$, where $\psi$ is a smooth cutoff centered on $\vy_0$ of sufficiently small support.
\end{corollary}

Let $L_{\vy_0}$ and $\psi$ be as in Theorem~\ref{thm:precise}, and let $\mathcal{T}$ be as in \eqref{equ_sum}. By Theorem~\ref{log_sing}, and using similar arguments to the proof above, we have the same inclusion as in the single tangency case:
\be
R^* (\psi \CP \mathcal{T}) \in I^{-\frac{9}{4} + p}(N^*\s).
\ee
Thus, the singularities in the reconstruction which correspond to those in $\WF(f)$ (i.e., the desired singularities) are stronger than the artifacts in $N^*\CL_a$. It has been established in the literature that the artifacts in $\Lambda$ (i.e., due to the failure of Bolker) are the same strength as those visible in $\WF(f)$ in the reconstruction \cite[section 8]{flu}. Therefore, the singularities in the reconstruction can be broken into three categories, those singularities which are visible in $\WF(f)$, the artifacts induced by $\Lambda$, and weaker artifacts in $N^*\CL_a$ caused by the non-linearities in the data.

For example, let us set $p = 1 = -\mathcal{O}(R^*R)$, as is customary so the singularities in the reconstruction which align with $WF(f)$ are of the same strength as the singularities of $f$. Then, in the scale of Sobolev spaces, and following Definition~\ref{def:SobolevWF}, the singularities in the reconstruction in $\WF(f)\cap \mathcal{V}$ or $\Lambda$ are in $H^{\alpha}$ microlocally for $\alpha < 1/2$ (not including $\alpha = 1/2$), while the singularities in $N^*\CL_a$ are in $H^{\beta}$ for $\beta < 3/2$, i.e., the singularities in the reconstruction due to non-linearities are one derivative smoother than the remaining singularities.

\section{Ridge Singularities of $\s$}\label{sec:folds}

Now that we have established our central theory which deals with smooth $\s$, we generalize these ideas to non-smooth $\s$, specifically when $\s$ adopts a ``ridge," which can be described using the intersection of two transversally intersecting smooth manifolds, similar to the edges of a cube. Such ridges occur in application, and are thus important to consider. The theory we present here uses many of the same ideas as the previous sections when $\s$ was smooth, and we reformulate these ideas in this section to address such non-smooth $\s$. We first define the notion of ``generalized tangency."

\begin{definition}[Generalized tangency] \label{def:gen tan} We say that $L_{\vy}$ is tangent to $\s$ at $\vx$ in the generalized sense if there exists $\xi\in N_{\vx}^*\s\setminus\{0\}$ such that $\xi\Theta=0$. 
\end{definition}

Clearly, the above definition reduces to the conventional tangency if $\s$ is smooth near $\vx$. The idea of this definition is that it generalizes the notion of tangency to the cases where $\s$ is not necessarily smooth at $\vx$, and when $\mu$ has singularities in multiple directions at a single point. Before, the $\mu$ we considered were conormal distributions associated with codimension-one surfaces. Hence, at each point the fiber of the conormal bundle was one-dimensional. In what follows, 
{when we say $L_\vy$ is tangent to $\s$, this now includes the notion of generalized tangency, not just conventional.} The boundaries of the domains $\pa\Omega_i$ may be nonsmooth along curves where two smooth faces of $\pa\Omega_i$ meet. We call these curves ``ridges." These ridges are analogous to the edges of a cube. Consider one such ridge, denoted $\ed$. Suppose $\tau$ is locally given by the transverse intersection of two smooth surfaces with defining functions $G_k(\vx)=0$, $k=1,2$. We assume that $\dd G_1(\vx)$ and $\dd G_2(\vx)$ are linearly independent when $\vx\in\ed$. The set of singularities of $Rf$ associated with the ridge, which we denote $\Gamma$ as in the previous sections, is determined by the condition that $L_{\vy}$ intersects $\ed$: 
\be\label{Ga def edge}
\Gamma=\{\vy\in B_{\vy_0}(\delta):\ L_{\vy}\cap\ed\neq\varnothing\}.
\ee
It is easy to see that in this case, $L_{\vy}$ is tangent to $\s$ in the generalized sense. This is because the {conormal} bundles to $L_\vy$ and $\s$ at $\vx$ in this case are both planes through $\vx$, and hence must intersect generating a singularity in the data. An equation for $\Gamma$ is thus obtained by solving the equations
\be\label{edge rt}
G_k(y(s)+t\Theta)=0,\ k=1,2.
\ee
These equations allow us to eliminate $t$ and obtain the relation between $s$ and $\Theta$.

The following assumption is a counterpart of the Assumption~\ref{ass:sngl} in the single (generalized) tangency case.

\begin{assumption}\label{ass:edge}$ $
Suppose $L_{\vy_0}$ intersects an ridge $\ed$ of $\s$ at $\vx_0$.
\begin{enumerate}
\item\label{transv int} $\tau$ is locally given by the transverse intersection of two smooth surfaces $G_k(\vx)=0$, $k=1,2$, i.e. $\dd G_1(\vx_0)$ and $\dd G_2(\vx_0)$ are linearly independent.
\item\label{GTheta} $\dd G_k(\vx_0)\Theta_0\neq0$, $k=1,2$.
\item\label{lin indep 3} The vectors $y^\prime(s_0),\Theta_0,\ed^\prime(\vx_0)$ are linearly independent.
\end{enumerate}
\end{assumption}

\begin{lemma}\label{lem:Gamma edge} Suppose Assumption~\ref{ass:edge} holds  and $\vx_0$ is the only point where $L_{\vy_0}$ is tangent to $\s$. Then $\Gamma$ is smooth and can be parameterized by $s = \phi(\Theta)$, where $\phi$ is smooth. Also, $\eta_s\neq0$ for any $\eta\in N_{\vy}^*\Gamma\setminus\{0\}$, $\vy\in B_{\vy_0}(\de)$.
\end{lemma}

\begin{proof}
Computing the Jacobian matrix of the system \eqref{edge rt}, we see that it equals
\be\label{jacob edge}
\begin{pmatrix} \dd G_1 \\ \dd G_2\end{pmatrix}\begin{pmatrix} y^\prime & \Theta \end{pmatrix}.
\ee

Let $e_k$, $k=1,2,3$, be nonzero vectors such that $G_ie_j=\de_{ij}$, $\forall i,j\in\{1,2\}$, and $G_ke_3=0$, $k=1,2$. Clearly, as $e_3$, we can take $\ed^\prime$, the tangent vector to $\ed$ at $\vx_0$. Then $y^\prime=\sum_k a_ke_k$ and $\Theta=\sum_k b_ke_k$. The value of the determinant is 
$a_1b_2-a_2b_1$. Suppose the determinant is zero. Then $a_1b_2=a_2b_1$, so the vectors $a_1e_1+a_2e_2$ and $b_1e_1+b_2e_2$ are parallel. This, in turn, implies that the vectors $y^\prime,\Theta,\ed^\prime$ are linearly dependent. This contradiction with Assumption~\ref{ass:edge}\eqref{lin indep 3} implies that the matrix in \eqref{jacob edge} is nondegenerate. Thus, like in the convex boundary case, $\Gamma$ is locally parameterized by $s=\phi(\Theta)$ for some smooth $\phi$.

Next we compute a generic vector $\eta\in N_{\vy}^*\Gamma\setminus\{0\}$, where $\vy\in\Gamma$. Denote $H_k(\vy,t):=G_k(y(s)+t\Theta)$. To find $\eta$ we solve $H_2(\vy,t)=0$ for $t$ in terms of $\vy$ and substitute into $H_1$. Then
\be\label{eta edge}\begin{split}
&\eta=\dd_{\vy}H_1-\frac{\pa_t H_1}{\pa_t H_2}\dd_{\vy} H_2=\Big(\dd G_1-\frac{\dd G_1\Theta}{\dd G_2\Theta}\dd G_2\Big)(y^\prime\ \Theta_\al^\top\ \Theta_\bt^\top),
\end{split}
\ee
where the last factor is a $3\times 3$ matrix. Note that $\dd G_2\Theta\neq0$ by Assumption~\ref{ass:edge}\eqref{lin indep 3}.The $\eta_s$ component of $\eta$ arises via the application of the covector in parentheses to $y^\prime$. Then $\eta_s\neq0$ provided that
\be\label{etas zero}
(\dd G_1 y^\prime)(\dd G_2\Theta)-(\dd G_1\Theta)(\dd G_2y^\prime)\not=0.
\ee
Using the vectors $e_k$, $k=1,2,3$, as above, we see that the expression in \eqref{etas zero} is exactly $a_1b_2-a_2b_1$ which we showed to be nonzero.  
\end{proof}

Our argument shows that when we consider a ray $L_{\vy_0}$ passing through an ridge $\ed$ of $\s$, the noncriticality condition is that the triple $y^\prime(s_0),\Theta_0,\ed^\prime(\vx_0)$ be linearly independent. The latter is a counterpart to the condition that $\Pi(\vx_0,\xi)$ not be tangent to $\ga$ at $y(s_0)$ as considered previously in the case of smooth $\s$ (see Assumption~\ref{ass:sngl}\eqref{sng transv}). In all cases, however, we have $\eta_s\not=0$ if $\eta\in N_{\vy}^*\Gamma$.

Similarly to Lemma~\ref{lem:sing-ty}, we easily obtain the following result.

\begin{lemma}\label{lem:edge sing} Suppose Assumption~\ref{ass:edge} holds  and $\vx_0$ is the only point where $L_{\vy_0}$ is tangent to $\s$. One has 
\be
\label{dist_R edge}
R\mu_E(\vy) =  g_1(\vy,E) + g_2(\vy,E)(s - \phi(\Theta))_+,\ \vy\in B_{\vy_0}(\delta),
\ee
where $g_1$ and $g_2$ are smooth; $g_2(\vy,E) \not\equiv 0$, $\vy\in\tau$; $s = \phi(\Theta)$ is the local equation of $\Gamma$ near $\vy_0$; and $\phi$ is smooth.  
\end{lemma}

\begin{proof}
Let $t_k(\vy)$ determines the point where the ray $L_{\vy}$ intersects the $k$th face, $k=1,2$. Thus, $t_k$ solves $G_k(y(s)+t_k\Theta)=0$. In particular, $t_1(\vy)=t_2(\vy)$ if $\vy\in\ed$. As is easily seen, up to a locally smooth function, the value of the ray transform can be written as follows
\be
Rf(\vy)\sim\int_{t_1(\vy)}^{t_2(\vy)} f(y(s)+t_k\Theta)\dd t,\ t_1(\vy)<t_2(\vy);\quad
Rf(\vy)\sim 0,\ t_1(\vy)\ge t_2(\vy).
\ee
Using that $f$ is locally smooth, Assumption~\ref{ass:edge}\eqref{GTheta}, and the discussion preceding Lemma~\ref{lem:edge sing}, we conclude that
\be
\label{edge 1st}
Rf(\vy) =  g_1(\vy) + g_2^\prime(\vy)(t_2(\vy)-t_1(\vy))_+=  g_1(\vy) + g_2(\vy)(s - \phi(\Theta))_+,\ \vy\in B_{\vy_0}(\delta),
\ee
where $g_2^\prime(\vy_0)\neq0$ and $g_2(\vy_0)\neq0$ if $f(\vx_0)\neq0$. Here we have also used that $t_2(\vy)-t_1(\vy)=0$ is another local equation for $\ed$ and $\pa_s(t_2(\vy_0)-t_1(\vy_0))\neq 0$. Replacing $f(\vx)$ with $\mu(\vx,E)$ proves the lemma.
\end{proof}

Using Lemma \ref{lem:edge sing}, the derivation in Section~\ref{ssec:conorm} can be easily extended to the case of ridge singularities and we get the following analog of Corollary~\ref{log_sing} (using Proposition~\ref{prop_sqrt} to compute the order of the distribution). 

\begin{corollary}\label{cor:log_sing edge}
Suppose $L_{\vy_0}$ passes through an ridge $\ed$ of $\s$ and Assumption~\ref{ass:edge} holds. We have
\be
\label{sqrt_sing edge}
\mathcal{R}\mu(\vy) = h_0(\vy) + h_1(\vy) (s - \phi(\Theta))_+,\ \vy\in B_{\vy_0}(\delta),
\ee
where $h_0$ and $h_1$ are smooth, and $\delta$ is sufficiently small. Further,
\be\label{Rmu distr edge}
\mathcal{R}\mu \in \dot{I}^{-2}(\Gamma)=I^{-9/4}(N^*\Gamma)
\ee
locally near $\vy_0$.
\end{corollary}

To study the added singularities along more general double tangent rays, we extend the derivation in Section~\ref{sec:dbl tang}. In this more general setting, if $L_{\vy}$ is  a double tangent ray, then each point of tangency can be either conventional or generalized. 

Thus, two additional cases arise: (1) $L_{\vy_0}$ is tangent to one smooth patch of $\s$ in the conventional sense and one non-smooth ridge, and (2) $L_{\vy_0}$ passes through two ridges of $\s$. In all three cases  (including the conventional double tangency case in Section~\ref{sec:dbl tang}), the leading singularity of the data due to nonlinearity is analogous to \eqref{double_lin} and is given by:
\be
\label{double_lin gen}
\mathcal{N}(\vy) \sim (s - \phi_1(\Theta))_{\imath_1}^{r_1}(s - \phi_2(\Theta))_{\imath_2}^{r_2},\ \vy\in B_{\vy_0}(\delta),
\ee
where $\mathcal{N}$ represents the non-linear ``component" of the data which includes products of distributions, using the same notation as before in \eqref{cal_N}. Here $r_k=1/2$ in the case of a conventional tangency and $r_k=1$ in the case of an ridge, $k=1,2$. In the first case, $\Gamma_k$ corresponds to a smooth patch of $\s$, and in the second -- to a {ridge} $\ed\in\s$. 

Lemma~\ref{lem:Gamma edge} and Assumption~\ref{ass:sngl} imply that if $L_{\vy_0}$ is a double tangent ray, then, locally, (1) $\Gamma=\Gamma_1\cup\Gamma_2$, where $\Gamma_1$ and $\Gamma_2$ are smooth, (2) $\vy_0\in\Gamma_1\cap\Gamma_2$, and (3) $\eta_s\not=0$ if $\eta\in N_{\vy}^*\Gamma_k$ for at least one of the $k$. This holds regardless of the nature of each tangency (conventional or generalized) at $\vx_1$ and $\vx_2$. 

The first two claims are obvious. Consider claim (3). If both tangencies are conventional, Assumption~\ref{ass:sngl} implies that $\eta_s^{(k)}\not=0$, $k=1,2$, and Lemma~\ref{lem:bolk fail} asserts that the intersection is transverse. If one of the tangencies is generalized, say $\Gamma_1$, then $\eta_s^{(1)}\not=0$ by Lemma~\ref{lem:Gamma edge}. If $\Gamma_2$ corresponds to a conventional tangency and $\eta_s^{(2)}=0$, then $\eta^{(1)}$ and $\eta_s^{(2)}$ are not parallel, and the intersection is transversal. If $\Gamma_2$ corresponds to a conventional tangency and $\eta_s^{(2)}\not=0$, Lemma~\ref{lem:bolk fail} applies. If $\Gamma_2$ corresponds to a generalized tangency, then $\eta_s^{(2)}\not=0$ and Lemma~\ref{lem:bolk fail} again applies. 

Thus, when working locally near such $\vy_0$, the non-linear component of the data behaves microlocally like a product of two conormal distributions with singularities on transversally intersecting manifolds. This is similar to the previous sections, although the order of the distributions changes depending on the type of edge. With this established, analogously to \eqref{cal_N}, we have
\be
\label{cal_N v2}
\mathcal{N}\in I^{-(r_1+\frac54),-(r_2+\frac12)}\big(N^*(\Gamma_1\cap\Gamma_2), N^*\Gamma_1\big) + I^{-(r_2+\frac54),-(r_1+\frac12)}\big(N^*(\Gamma_1\cap\Gamma_2), N^*\Gamma_2\big),
\ee
locally near $\vy_0$, and we have calculated the orders of the distributions using \eqref{u1u2} with $\nu_k=-(r_k+\tfrac54)$. Note how the orders decrease (i.e., the singularities become smoother) when the double tangency involves ridges.

Lemma~\ref{lem_clean} only requires that $\Gamma_1$, $\Gamma_2$ be smooth and the intersection $\Gamma_1\cap\Gamma_2$ be transversal. As we established, this holds in all three cases. Therefore, in all three cases, the added singularity in the reconstructed image is described similarly to \eqref{R* conorm}: 
\be\label{R* conorm gen}
R^* (\psi \mathcal{P}\mathcal{N}) \in I^{-(r_1+r_2+\frac94) + p}\big(N^*\CL_a\setminus N^*\s\big),
\ \CL_a:=\cup_{\vy\in\Gamma_1\cap\Gamma_2}\big(L_{\vy}\cap X\big).
\ee

As before, $\CL_a$ is a smooth surface. Generically, $\CL_a$ can be parameterized by $s$. However, in this more general setting, the details are somewhat cumbersome and not illuminating, so we do not present them here. For our purposes, Lemma~\ref{lem_clean} and the result \eqref{R* conorm gen} are sufficient.

Similarly, Theorem~\ref{thm:C nonlin_art} can be generalized to the case of more general tangency, but we do not do it here as it does not provide much insight beyond that in the proof of Theorem~\ref{thm:C nonlin_art}.   

\section{Experiments}\label{sec:exp}
In this section, we present simulated experiments to illustrate our theory. We consider a circular scanning geometry where $\gamma$ is the circle $y(s) = r(\cos s, \sin s, 0)$ with radius $r = 50$. The reconstruction domain is $X = B_0(45)$, i.e., the ball centered at zero with radius 45, and we sample $\vy = (s, \Theta)$ for $s \in [0,2\pi]$ and for rays with direction $\Theta$ which intersect the detector plane $P_s = \{\vx \cdot y(s) = -r\}$ as is common in cone-beam CT \cite{sidky2008image}. The detector plane has size $[-r,r]^2$. Let $(p_1,p_2) \in [-r,r]^2$ be coordinates on the detector plane. Then, the data, $\mathcal{R}\mu(s,p_1,p_2)$, is sampled on the cube $[0,2\pi]\times[-r,r]^2$.
\begin{figure}[!h]
\centering
\begin{subfigure}{0.27\textwidth}
\includegraphics[width=0.9\linewidth, height=3.5cm, keepaspectratio]{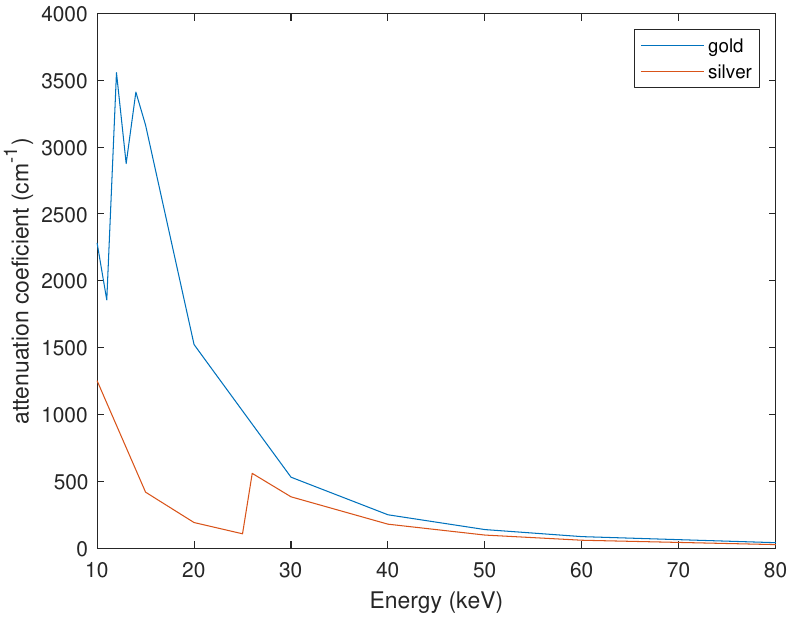}
\end{subfigure}
\begin{subfigure}{0.27\textwidth}
\includegraphics[width=0.9\linewidth, height=3.5cm, keepaspectratio]{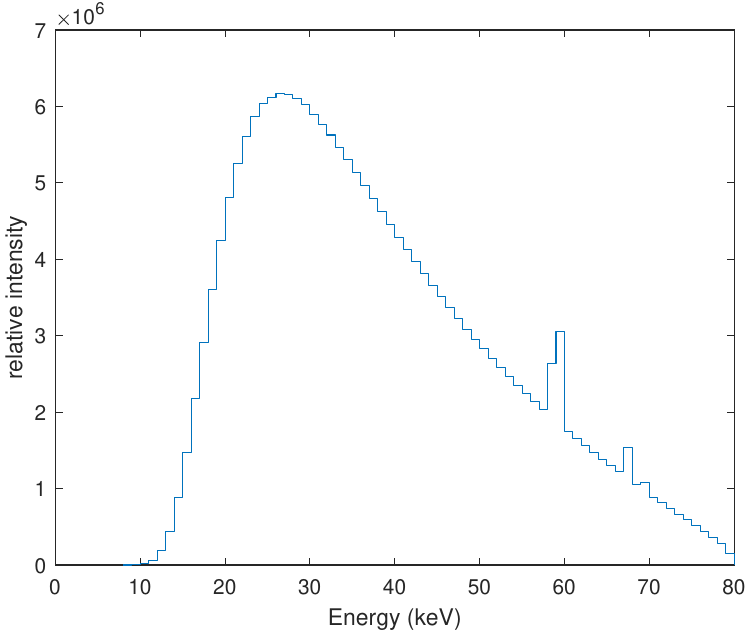}
\end{subfigure}
\caption{Attenuation coefficients of Gold and Silver (left), and an 80kVp spectrum (right). In this case, $E_{\text{mx}} = 80$keV.}
\label{F1}
\end{figure}

To simulate data, we first discretize $R$ and store it as a sparse matrix. To calculate the ray-pixel intersections, we use the ray-tracing code supplied in \cite{jacobs1998fast}.  We set $\mu(\vx,E) = \mu_g(E)\chi_{B_1}(\vx) + \mu_s(E)\chi_{B_2}(\vx)$ and $f = \rho_g\chi_{B_1}(\vx) + \rho_s\chi_{B_2}(\vx)$, where $\mu_g, \mu_s$ and $\rho_g,\rho_s$ are the attenuation coefficients and densities of gold and silver, respectively. While $f$ is not needed to generate the non-linear data, in the below experiments we compare reconstructions using linear data and non-linear data, and $f$ is used to simulate the linear data via $Rf$. The attenuation coefficients and density values were obtained from the NIST database \cite{hubbell2004tables}. 
\begin{figure}[!h]
\centering
\includegraphics[width=0.9\linewidth, height=6cm, keepaspectratio]{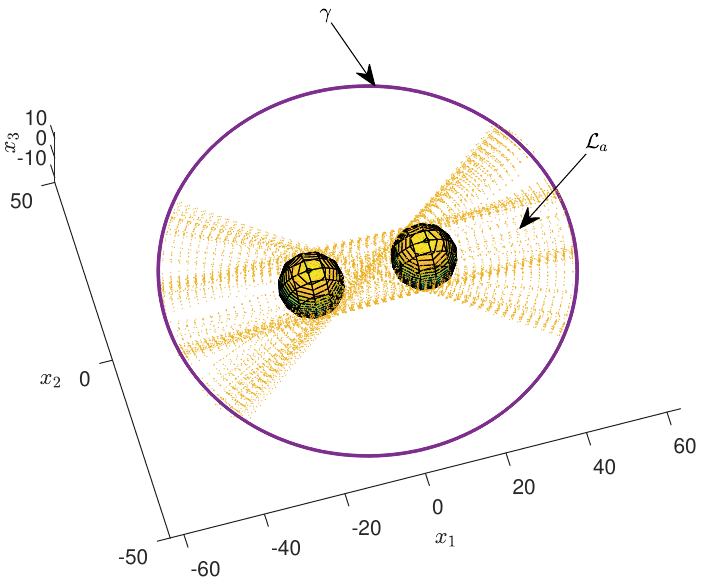}
\caption{Illustration of scanning geometry. $\gamma$ is the purple ring and $B_1$ and $B_2$ are the meshed balls. The dashed orange surface is $\mathcal{L}_a$ which is formed of double tangent rays to $B_1$ and $B_2$ which also pass through $\gamma$.}
\label{F1_1}
\end{figure}
To simulate $P(E)$, we use the ``SpekCalc" software \cite{poludniowski2009spekcalc}, using an 80kVp spectrum with a Tungsten target. For the purposes of this simulated example, the detector efficiency is 100\%. See Figure \ref{F1} for plots of $\mu_g,\mu_s$ and $P(E)$. The balls $B_1$ and $B_2$ have centers $(14,0,0)$ and $(-14,0,0)$, respectively, and both have radius 8. See Figure \ref{F1_1} for an illustration of our scanning geometry. The double tangent artifact surface, $\CL_a$, is also highlighted. To simulate $\mathcal{R}\mu$, we use \eqref{model_data}.

\begin{figure}[!h]
\centering
\begin{subfigure}{0.24\textwidth}
\includegraphics[width=0.9\linewidth, height=3.2cm, keepaspectratio]{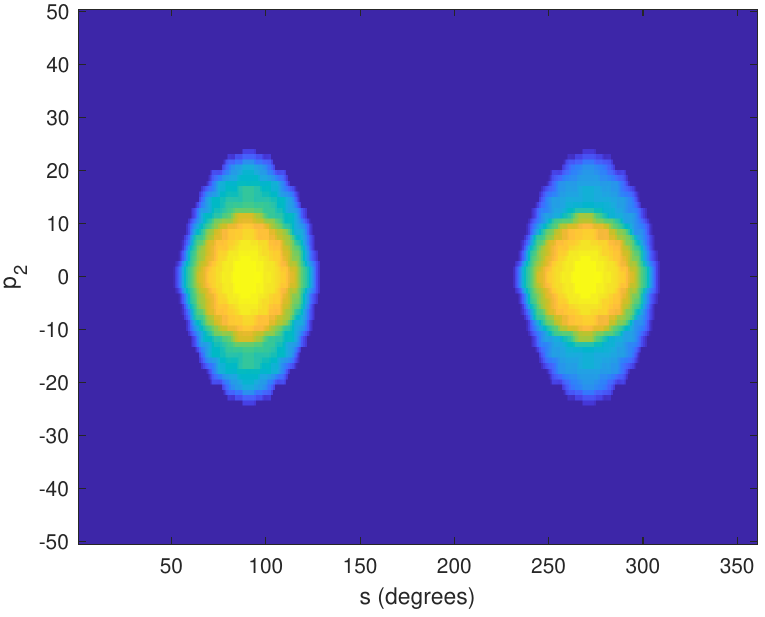}
\end{subfigure}
\begin{subfigure}{0.24\textwidth}
\includegraphics[width=0.9\linewidth, height=3.2cm, keepaspectratio]{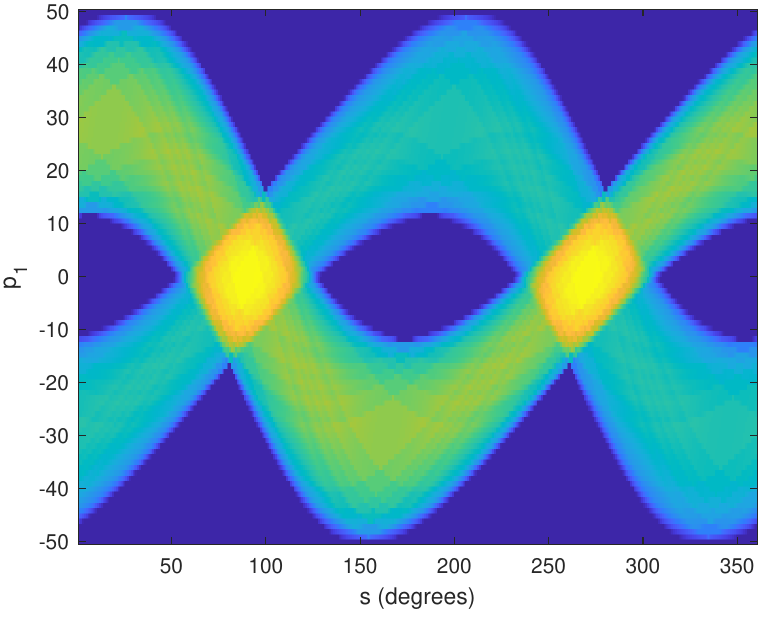}
\end{subfigure}
\begin{subfigure}{0.24\textwidth}
\includegraphics[width=0.9\linewidth, height=3.2cm, keepaspectratio]{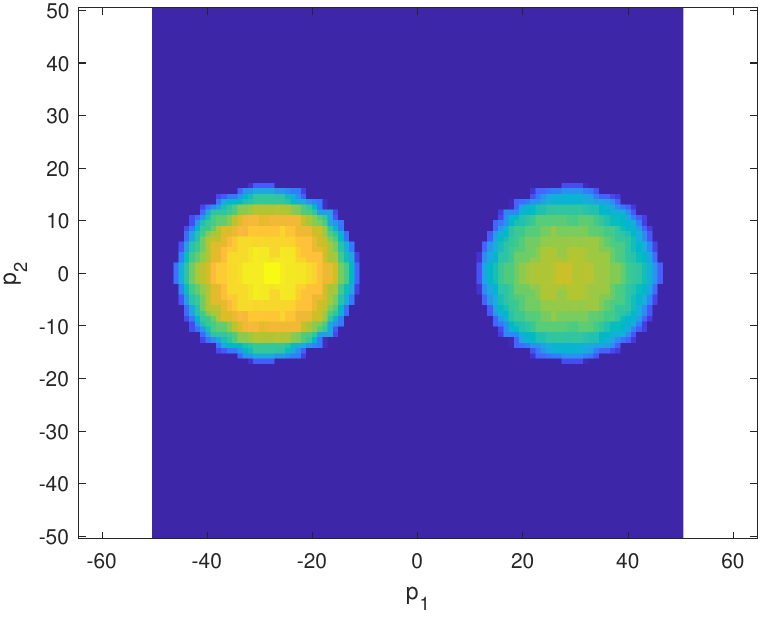}
\end{subfigure}
\\
\begin{subfigure}{0.24\textwidth}
\includegraphics[width=0.9\linewidth, height=3.2cm, keepaspectratio]{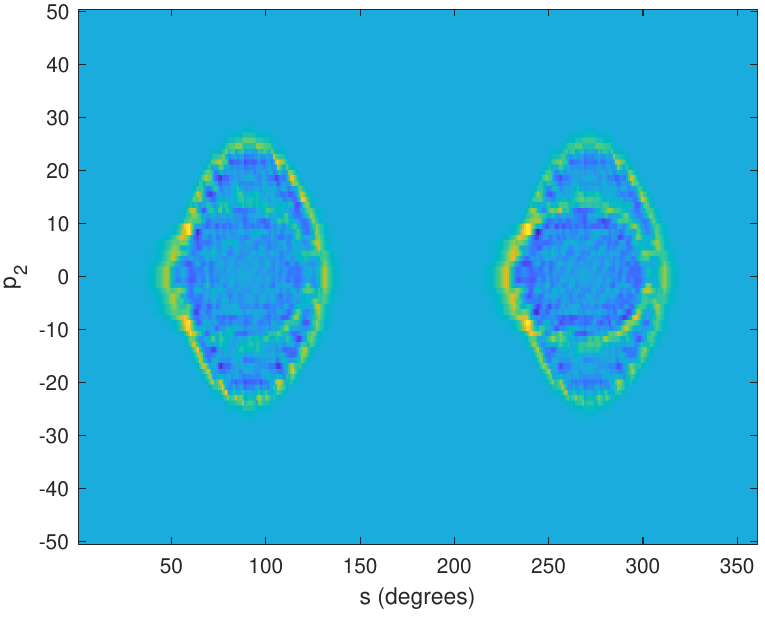}
\subcaption*{$\{p_1 = 0\}$}
\end{subfigure}
\begin{subfigure}{0.24\textwidth}
\includegraphics[width=0.9\linewidth, height=3.2cm, keepaspectratio]{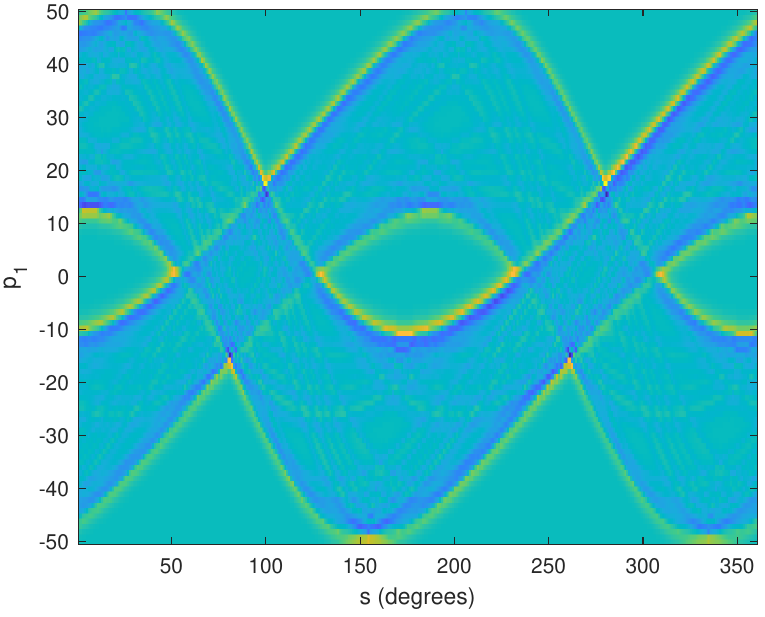}
\subcaption*{$\{p_2 = 0\}$}
\end{subfigure}
\begin{subfigure}{0.24\textwidth}
\includegraphics[width=0.9\linewidth, height=3.2cm, keepaspectratio]{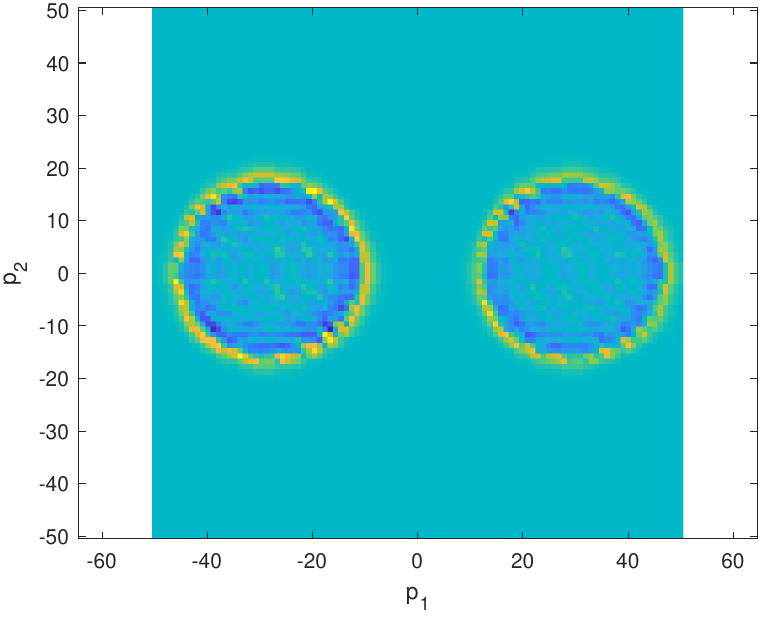}
\subcaption*{$\{s = \pi\}$}
\end{subfigure}
\caption{Simulated data $\mathcal{R}\mu(s,p_1,p_2)$ on three central slices. The top row is the original data $\mathcal{R}\mu$, and the bottom row shows the filtered data $\mathcal{P} \mathcal{R}\mu$.}
\label{F2}
\end{figure}

For {FBP} reconstruction, we apply $R^*\mathcal{P}$ to the data, where $\mathcal{P} = -\Delta_{P_s}$ is the negative Laplacian applied in Cartesian coordinates on the detector plane. Example simulated data is shown in figure \ref{F2}. We show the original data and the data filtered by $\mathcal{P}$. We also consider the Landweber method \cite{hansen2018,landweber1951iteration}, which has been shown in previous publications to highlight artifacts predicted by microlocal theory \cite{webber2023ellipsoidal}.
\begin{figure}[!h]
\centering
\begin{subfigure}{0.24\textwidth}
\includegraphics[width=0.9\linewidth, height=3.2cm, keepaspectratio]{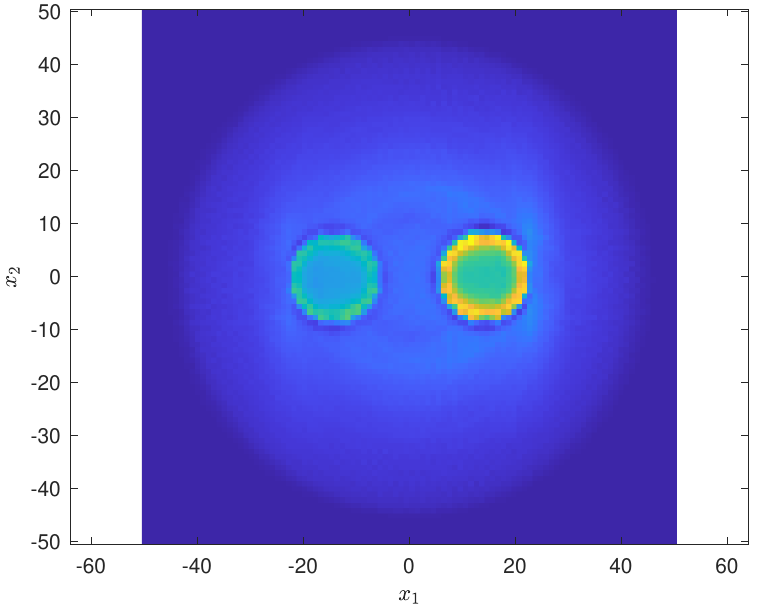}
\end{subfigure}
\begin{subfigure}{0.24\textwidth}
\includegraphics[width=0.9\linewidth, height=3.2cm, keepaspectratio]{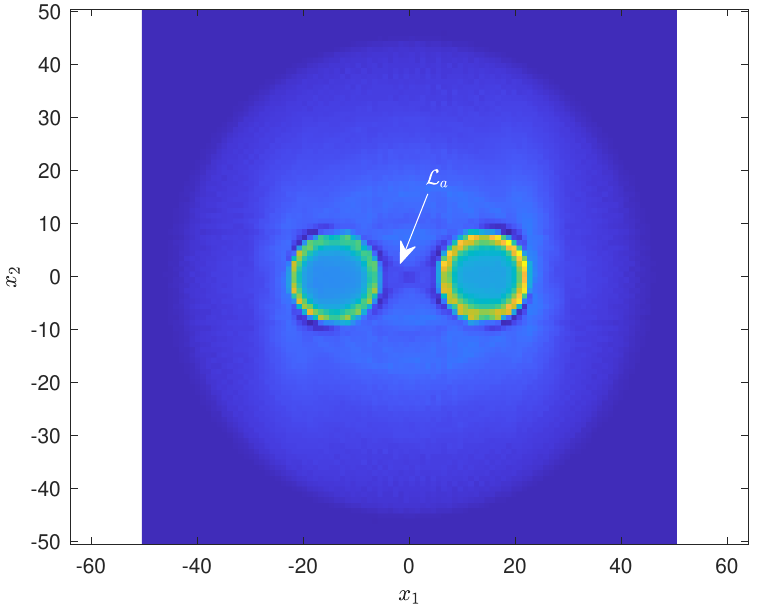}
\end{subfigure}
\begin{subfigure}{0.24\textwidth}
\includegraphics[width=0.9\linewidth, height=3.2cm, keepaspectratio]{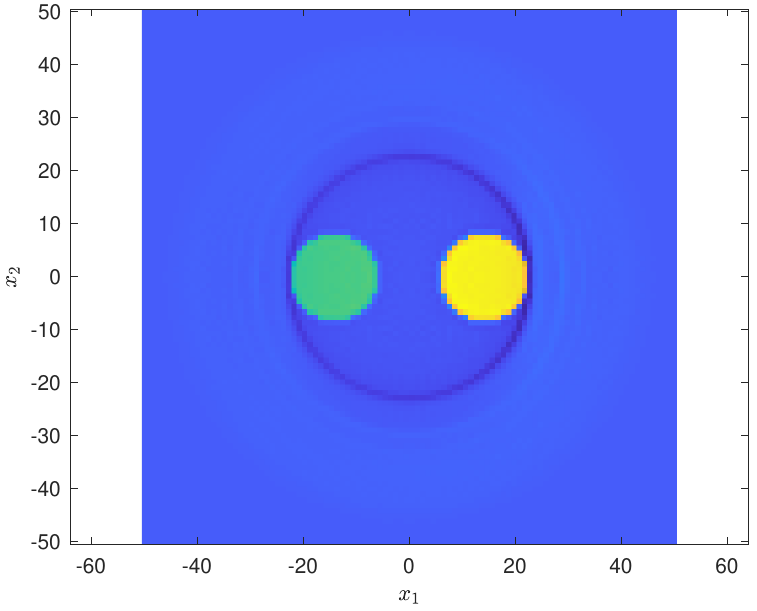}
\end{subfigure}
\begin{subfigure}{0.24\textwidth}
\includegraphics[width=0.9\linewidth, height=3.2cm, keepaspectratio]{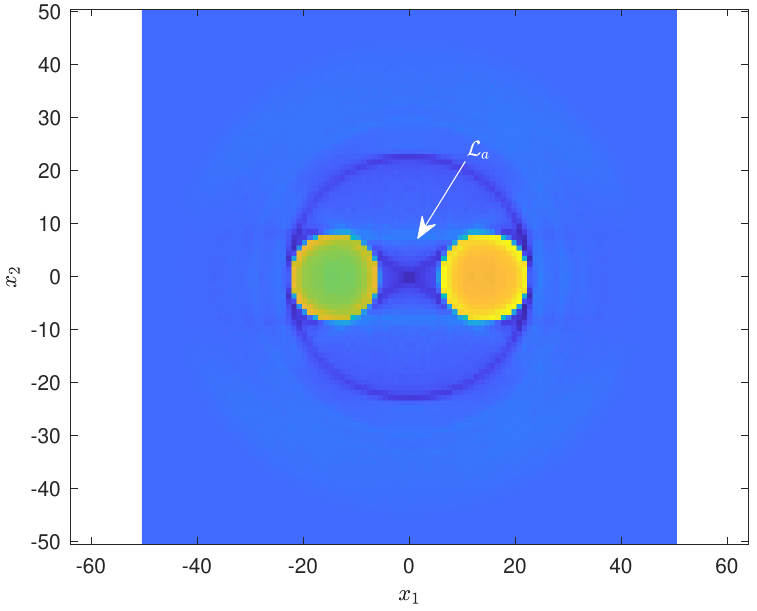}
\end{subfigure}
\\
\begin{subfigure}{0.24\textwidth}
\includegraphics[width=0.9\linewidth, height=3.2cm, keepaspectratio]{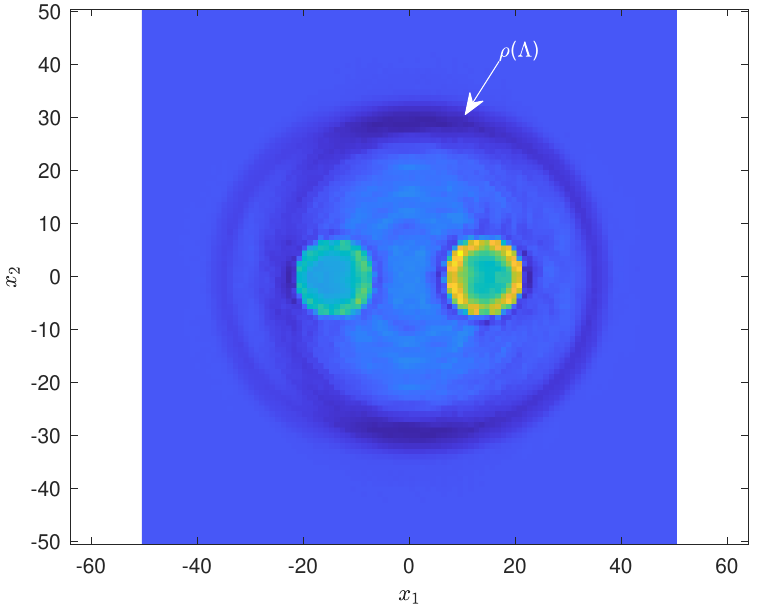}
\subcaption*{$R^*\mathcal{P}Rf$}
\end{subfigure}
\begin{subfigure}{0.24\textwidth}
\includegraphics[width=0.9\linewidth, height=3.2cm, keepaspectratio]{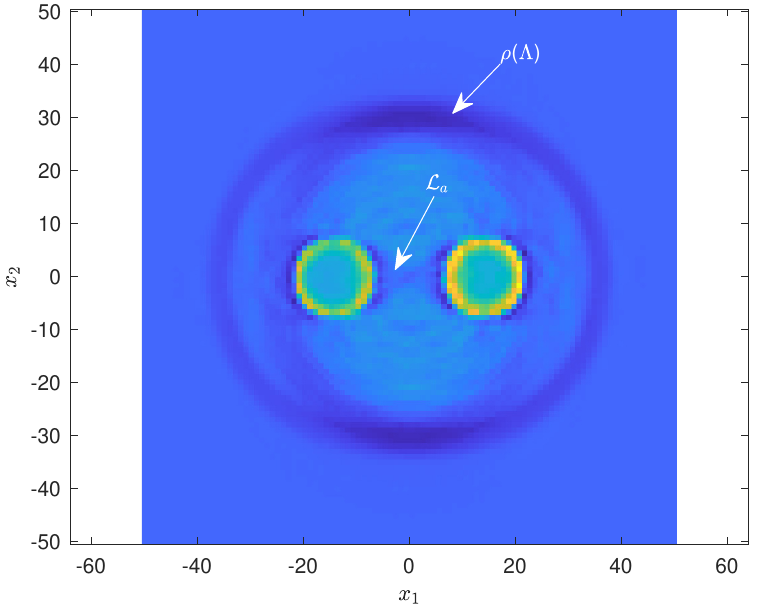}
\subcaption*{$R^*\mathcal{P}\mathcal{R}\mu$}
\end{subfigure}
\begin{subfigure}{0.24\textwidth}
\includegraphics[width=0.9\linewidth, height=3.2cm, keepaspectratio]{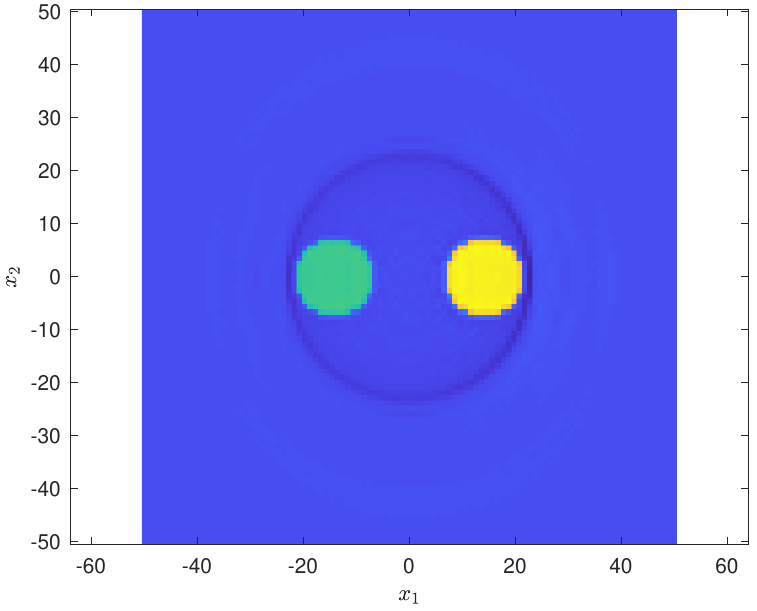}
\subcaption*{Landweber (linear)}
\end{subfigure}
\begin{subfigure}{0.24\textwidth}
\includegraphics[width=0.9\linewidth, height=3.2cm, keepaspectratio]{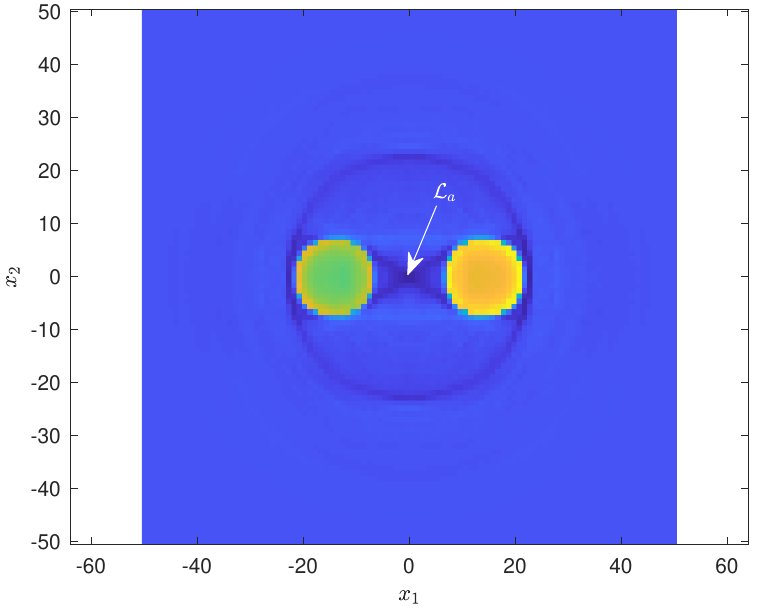}
\subcaption*{Landweber (non-linear)}
\end{subfigure}
\caption{Reconstructions of gold and silver balls showing the artifacts induced by $\Lambda$ and $N^*\CL_a$. The top and bottom rows show different slices at $\{x_3 = 0\}$ and $\{x_3 = 4\}$, respectively. The linear and non-linear reconstructions are shown side-by-side for comparison. We show reconstructions using the FBP (first two columns) and the Landweber method (last two columns). The artifacts due to $\Lambda$ are visible on the bottom row. The artifacts on $\mathcal{L}_a$ appear as a dark cross and two light parallel streaks tangent to both balls. Here, $\rho : T^*X \to X$ denotes the natural projection.}
\label{F3}
\end{figure}

In figure \ref{F3}, we present our image reconstructions on slices $\{x_3 = 0\}$ and $\{x_3 = 4\}$. We show reconstructions using linear data (i.e., $Rf$) and non-linear data ($\mathcal{R}\mu$) side-by-side, for comparison. In the non-linear FBP reconstructions, we see artifacts along streaks inbetween $B_1$ and $B_2$, which are not visible in the linear reconstructions. These are cross-sections of the dashed orange surface in Figure \ref{F1_1}. The artifacts are fainter than those in $\WF(f)$, which is as expected, since the non-linear artifacts are one degree smoother on Sobolev scale. In the $\{x_3 = 4\}$ plane slices, we can also see the artifacts induced by $\Lambda$ (i.e., the artifacts due to Bolker), which are present in the linear and non-linear reconstructions, as expected. In this example, the natural projection of $\Lambda$ onto $X$ is the union of two wide open cones tangent to $B_k$ and the source trajectory, $k=1,2$. The dark rings in the FBP $\{x_3 = 4\}$ slice reconstructions are cross-sections of these cones (elllipses). The artifacts induced by $\Lambda$ are notably stronger than those on $\CL_a$, which is in line with our predictions. 

In the Landweber reconstructions, the non-linear artifacts are more clearly visible, and the Landweber method appears to suppress the artifacts due the Bolker condition violations as they are not apparent in the reconstructions. We notice a ring artifact in the Landweber reconstructions, which looks tangent to both balls. As this appears in both the linear and non-linear reconstruction, it is clearly not related to non-linearities and is likely an artifact related to the Landweber iteration. Such artifacts are not the purpose of this paper and we provide the Landweber reconstructions to show how non-linear artifacts affect iterative reconstruction.

\section*{Acknowledgments}
JWW wishes to acknowledge funding support from The Cleveland Clinic Foundation, The Honorable Tina Brozman Foundation, the V Foundation, and the National Cancer Institute R03CA283252-01.

\bibliographystyle{abbrv}
\bibliography{refs, My_Collection, RefRevolution}

\end{document}